\documentclass[11pt]{amsart}

\usepackage{graphicx}
\usepackage{epsfig}

\usepackage{amsmath}
\usepackage{amssymb}
\usepackage{bigints}
\usepackage{tikz}
\usetikzlibrary{arrows}

\theoremstyle{plain}
\newtheorem{theorem}                {Theorem}      [section]
\newtheorem*{theorem*}                {Theorem \ref{thm:appl}}
\newtheorem{proposition}  [theorem]  {Proposition}

\newtheorem{lemma}        [theorem]  {Lemma}

\theoremstyle{definition}

\newtheorem{remark}       [theorem]  {Remark}

\DeclareMathOperator{\trace}{trace} 

\DeclareMathOperator{\Div}{div} 
 
\DeclareMathOperator{\ricci}{Ricci}

\DeclareMathOperator{\Span}{span}

\DeclareMathOperator{\grad}{grad}

\numberwithin{equation}{section}

\begin{document}

\title[Complete biconservative surfaces in $\mathbb{R}^3$ and $\mathbb{S}^3$]
{Complete biconservative surfaces in $\mathbb{R}^3$ and $\mathbb{S}^3$}

\author{Simona~Nistor}

\address{Faculty of Mathematics\\ Al. I. Cuza University of Iasi\\
Bd. Carol I, 11 \\ 700506 Iasi, Romania} \email{nistor.simona@ymail.com}

\thanks{The work was supported by a grant of the Romanian National Authority for Scientific Research and Innovation, CNCS - UEFISCDI, project number PN-II-RU-TE-2014-4-0004.}

\subjclass[2010]{Primary 53A10; Secondary 53C40, 53C42}

\keywords{Biconservative surfaces, complete surfaces, mean curvature function, real space forms}

\begin{abstract}
In this paper we consider the complete biconservative surfaces in Euclidean space $\mathbb{R}^3$ and in the unit Euclidean sphere $\mathbb{S}^3$. Biconservative surfaces in 3-dimensional space forms are characterized by the fact that the gradient of their mean curvature function is an eigenvector of the shape operator, and we are interested in studying local and global properties of such surfaces with non-constant mean curvature function. We determine the simply connected, complete Riemannian surfaces that admit biconservative immersions in $\mathbb{R}^3$ and $\mathbb{S}^3$. Moreover, such immersions are explicitly described.
\end{abstract}

\maketitle
\section{Introduction}

The study of \textit{biconservative submanifolds} is derived from the theory of \textit{biharmonic submanifolds} which has been of large interest in the last decade (see, for example \cite{BMO12,BMO10a,BMO10,BMO08,CI91,O10,OW11}).

Let $\left(M^m,g\right)$ and $\left(N^n,h\right)$ be two Riemannian manifolds. A critical point of the \textit{bienergy functional}
$$
E_2:C^{\infty}(M,N)\rightarrow\mathbb{R},\quad E_{2}(\phi)=\frac{1}{2}\int_{M}|\tau(\phi)|^{2}\ v_g,
$$
where $\tau(\phi)$ is the tension field of a smooth map $\phi:M\to N$, is called a \textit{biharmonic map}, and it is characterized by the vanishing of the \textit{bitension field} $\tau_2(\phi)$ (see\cite{GYJ}).

A Riemannian immersion $\phi:M^m\to \left(N^n,h\right)$ or, simply, a submanifold $M$ of $N$, is called \textit{biharmonic} if $\phi$ is a biharmonic map.

In 1924, D. Hilbert called the \textit{stress-energy tensor} associated to a functional $E$, a symmetric 2-covariant tensor $S$ which is conservative, i.e., $\Div S=0$, at the critical points of $E$. In the case of the bienergy functional $E_2$, G. Y. Jiang defined in 1987 the stress-bienergy tensor $S_2$ and proved that it satisfies
$$
\Div S_2=-\langle\tau_2(\phi),d\phi\rangle.
$$
Thus, if $\phi$ is biharmonic, then $\Div S_2=0$ (see \cite{J}).

For biharmonic submanifolds, from the above relation, we see that $\Div S_2=0$ if and only if the tangent part of the bitension field vanishes. A submanifold $M$ is called \textit{biconservative} if $\Div S_2=0$.

The biconservative submanifolds were studied for the first time in 1995 by Th. Hasanis and Th. Vlachos (see \cite{HV95}). In that paper the biconservative hypersurfaces in the Euclidean space $\mathbb{R}^n$ were called \textit{H-hypersurfaces}, and they were fully classified in $\mathbb{R}^3$ and $\mathbb{R}^4$.

Recent results in the field of biconservative submanifolds were obtained, for example, in \cite{FOP,Fu,FT,MOR,S,UT}.

When the ambient space is a 3-dimensional space form $N^3(c)$, it is easy to see that the surfaces with constant mean curvature (\textit{CMC surfaces}) are biconservative. Therefore, we are interested in biconservative surfaces which \textit{are not CMC}, i.e., $\grad f\neq0$, where $f$ is the mean curvature function.

The explicit local parametric equations of biconservative surfaces in $\mathbb{R}^3$, $\mathbb{S}^3$, and $\mathbb{H}^3$ were determined in \cite{CMOP} and \cite{Fu}. When the ambient space is $\mathbb{R}^3$ the result in \cite{HV95} was also reobtained in \cite{CMOP}.

Our paper is organized as follows. In Section $\ref{preliminaries}$ we recall the results concerning the local classification of biconservative surfaces of non-constant mean curvature function in $\mathbb{R}^3$ and $\mathbb{S}^3$, as they are presented in ~\cite{CMOP}. Then, we recall a result about the intrinsic characterisation of biconservative surfaces in 3-dimensional space form $N^3(c)$ (see ~\cite{FNO}). More precisely, this result provides the necessary and sufficient conditions for an abstract Riemannian surface $\left(M^2,g\right)$ to admit, locally, a biconservative embedding with $|\grad f|> 0$ in $N^3(c)$. It is also recalled that, if a simply connected Riemannian surface $\left(M^2,g\right)$ admits a biconservative immersion with $|\grad f|> 0$ in $N^3(c)$, then it is unique.

In the second part of the paper, we take the next step and, writing the metric $g$ in isothermal coordinates as $g=e^{2\varphi}\left(du^2+dv^2\right)$, we determine the equation which must be satisfied by $\varphi$ such that $\left(M^2,g\right)$ can be locally embedded in $N^3(c)$ as a non CMC biconservative surface. This equation is then solved for $c=0$ and $c=1$ (Proposition $\ref{prop4.3}$ and Proposition $\ref{prop4.5}$).

Our main goal is to extend the \textit{local} classification results for biconservative surfaces in $N^3(c)$, with $c=0$ and $c=1$, to \textit{global} results, i.e., we ask that biconservative surfaces to be \textit{complete} and with $|\grad f|>0$ on an open dense subset.

Our first main result is Theorem $\ref{main_th1}$ where we determine the simply connected complete Riemannian surfaces $\left(\mathbb{R}^2,g_C\right)$ which admit a biconservative immersion in $\mathbb{R}^3$. Moreover, these immersions are explicitly given and they have $|\grad f|> 0$ on an open dense subset of $\mathbb{R}^2$.

Next, we obtain a similar result when $c=1$. In Theorem $\ref{main_th2}$ we determine the simply connected complete Riemannian surfaces  $\left(\mathbb{R}^2,g_{C,C^\ast}\right)$ which admit a biconservative immersion in $\mathbb{S}^3$. We show that, up to isometries, there exists only a one-parameter family of such Riemannian surfaces indexed by $C$. In order to prove Theorem $\ref{main_th2}$, the key ingredient is that a biconservative surface in $\mathbb{S}^3$ is locally isometric to a surface of revolution in $\mathbb{R}^3$ (Theorem $\ref{theorem3.18}$) and then, by a gluing process, we extend this surface of revolution, which is not complete, to a complete one (Theorem $\ref{main_th3}$). The new surface admits a biconservative immersion in $\mathbb{S}^3$, with $|\grad f|> 0$ on an open dense subset. Finally, we prove the uniqueness of the complete biconservative surfaces in $\mathbb{S}^3$.

\textbf{Conventions.} We denote an abstract Riemannian surface, or an abstract Riemannian manifold by $\left(M,g\right)$. To avoid any confusions, in the case of surfaces, we denote $S^2$ the image in the ambient space of an abstract Riemannian surface $\left(M^2,g\right)$ through the immersion $\phi$.

\textbf{Acknowledgments.} The author would like to thank C.~Oniciuc, D.~Fetcu and S.~Montaldo for many useful comments and suggestions.

\section{Preliminaries}\label{preliminaries}

We first recall two known results concerning the completeness of a Riemannian manifold (see \cite{DC,G}).

\begin{proposition}[\cite{G}]
\label{prop1}
Let $g$ and $\tilde{g}$ be two Riemannian metrics on a manifold $M$. If $(M,g)$ is complete and $\tilde{g}-g$ is non-negative definite at any point of $M$, then $(M,\tilde{g})$ is also complete.
\end{proposition}

\begin{proposition}[\cite{DC}]
Let $S^2$ be a regular surface in $\mathbb{R}^3$. If $S^2$ is a closed subset of $\mathbb{R}^3$, then $S^2$ is complete.
\end{proposition}

Concerning biharmonic maps, as we have already seen, the Euler-Lagrange equation for bienergy functional is given by $\tau_2(\phi)=0$, where
$$
\tau_{2}(\phi)=-\Delta\tau(\phi)-\trace R^N(d\phi,\tau(\phi))d\phi
$$
is the \textit{bitension field} of $\phi$, $\Delta=-\trace(\nabla^{\phi})^2 =-\trace(\nabla^{\phi}\nabla^{\phi}-\nabla^{\phi}_{\nabla})$ is the rough Laplacian defined on
sections of $\phi^{-1}(TN)$ and $R^N$ is the curvature tensor of $N$ given by $R^N(X,Y)Z=[\overline\nabla_X,\overline\nabla_Y]Z-\overline\nabla_{[X,Y]}Z$.

Now we consider the stress-energy tensor $S_2$ associated to the bienergy. This tensor, that was studied for the first time in \cite{J} and then in papers like \cite{CMOP,Fu,LMO,MOR,MOR2}, is given by
\begin{align*}
S_2(X,Y)=&\frac{1}{2}|\tau(\phi)|^2\langle X,Y\rangle+\langle d\phi,\nabla\tau(\phi)\rangle\langle X,Y\rangle\\&-\langle d\phi(X),\nabla_Y\tau(\phi)\rangle-\langle d\phi(Y),\nabla_X\tau(\phi)\rangle
\end{align*}
and it satisfies
$$
\Div S_2=-\langle\tau_2(\phi),d\phi\rangle.
$$
We can see that in the case when $\phi$ is a submersion, $\Div S_2$ vanishes if and only if $\phi$ is biharmonic. When $\phi:M\to N$ is an Riemannian immersion, then $(\Div S_2)^{\sharp}=-\tau_2(\phi)^{\top}$, where $\sharp$ denotes the musical isomorphism sharp. Therefore, in general, for a Riemannian immersion, $\Div S_2$ does not automatically vanish.

The biharmonic equation $\tau_2(\phi)=0$ of a submanifold $\phi:M\rightarrow N$ can be decomposed in its normal and tangent part, and in the particular case of hypersurfaces $M$ in $N$, one obtains the following theorem.

\begin{theorem}[\cite{BMO12,O10}]
If $M^m$ is a hypersurface in a Riemannian manifold $N^{m+1}$, then $M$ is biharmonic if and only if the tangent and normal components of $\tau_2(\phi)$ vanish, i.e., respectively
$$
{2A(\grad f)+f\grad f-2f(\ricci^N(\eta))^{\top}=0}
$$
and
$$
\Delta f+f|A|^2-f\ricci^N(\eta,\eta)=0,
$$
where $\eta$ is a unit normal vector field of $M$ in $N$ and $f=\trace A$ is the mean curvature function.
\end{theorem}

From this decomposition, it follows that a surface $\phi:M^2\to N^3(c)$ in a space form $N^3(c)$ is biconservative if and only if
\begin{equation*}
A(\grad f)=-\frac{f}{2}\grad f.
\end{equation*}

\subsection{Biconservative surfaces in $\mathbb{R}^3$}

$\newline$
In the following, we will present some results concerning biconservative surfaces with $|\grad f|>0$ in the 3-dimensional Euclidean space.

\begin{theorem}[\cite{CMOP}]
\label{th3.3}
Let $S^2$ be a biconservative surface in $\mathbb{R}^3$ with $f(p)>0$ and $(\grad f)(p)\neq 0$, at any $p\in M$. Then, locally, $S^2$ is a surface of revolution given by
$$
X_{C_1}(\rho,v)=\left(\rho \cos v,\rho \sin v, t_{C_1}(\rho)\right),
$$
where
$$
t_{C_1}(\rho)=\frac{3}{2C_1}\left(\rho^{1/3}\sqrt{C_1\rho^{2/3}-1}+\frac{1}{\sqrt{C_1}}\log \left(\sqrt{C_1}\rho^{1/3}+\sqrt{C_1\rho^{2/3}-1}\right)\right),
$$
$\rho>C_{1}^{-3/2}$, with $C_1$ a positive constant.
\end{theorem}

Obviously, $\lim_{\rho\searrow C_1^{-3/2}}{t_{C_1}(\rho)}=0$. As $t'_{C_1}(\rho)>0$ for any $\rho\in\left(C_{1}^{-3/2},\infty\right)$, we can think $\rho$ as a function of $t$ and
$$
X_{C_1}(t,v)=\left(\rho_{C_1}(t) \cos v, \rho_{C_1} (t)\sin v, t\right), \qquad t\in (0,\infty).
$$

\begin{proposition}[\cite{MOR3}]
\label{prop3.4}
If we consider the symmetry of the graph of $t_{C_1}$, when $\rho\in \left(C_{1}^{-3/2},\infty\right)$ with respect to the $O\rho=Ox$ axis, we get a smooth complete biconservative surface $\tilde{S}^2_{C_1}$ in $\mathbb{R}^3$, given by
$$
X_{C_1}(t,v)=\left(x_{C_1}(t)\cos v,x_{C_1}(t)\sin v,t\right), \qquad (t,v)\in \mathbb{R},
$$
where
\begin{equation*}
x_{C_1}(t)=\left\{
\begin{array}{ccc}
\rho_{C_1}(t)&,& t>0\\
C_1^{-3/2}&,& t=0\\
\rho_{C_1}(-t)&,& t<0
\end{array}
\right.
\end{equation*}
is a smooth function. Moreover, the curvature function $f$ is positive and $\grad f$ is different from zero at any point of an open dense subset of $\tilde{S}^2_{C_1}$.
\end{proposition}

Moreover the above construction of complete biconservative surfaces with $\grad f$ different from zero on an open dense subset is unique.

\begin{proposition}
\label{prop2.6}
The complete biconservative surfaces $\tilde{S}_{C_1}$ are unique (up to reparameterization).
\end{proposition}

\begin{proof}
We denote by $S_{C_1}$ the biconservative surface defined by
\begin{eqnarray*}
X_{C_1}(\rho,v)&=&\left(\rho\cos v,\rho \sin v, t_{C_1}(\rho)\right)\\
&=&\rho\cos v\ \overline{e}_1+\rho\sin v\ \overline{e}_2+t_{c_1}(\rho)\ \overline{e}_3,
\end{eqnarray*}
where $t_{C_1}(\rho)$ is given in Theorem $\ref{th3.3}$. The boundary of $S_{C_1}$, i.e. $\overline{S}_{C_1}\setminus S_{C_1}$, is the circle
$$
\left(C_1^{-3/2}\cos v, C_1^{-3/2} \sin v, 0\right),
$$
which lies in the $xOy$ plane (a plane perpendicular to the rotation axis $Oz$).

At a boundary point, the tangent plane to the closure $\overline{S}_{C_1}$ of $S_{C_1}$ is parallel to $Oz$. Moreover, along the boundary, the mean curvature function is constant $f_{C_1}=\frac{2}{3C_1^{-3/2}}$ and $\grad f_{C_1}=0$. Thus, we can expect to ``glue'' along the boundary two biconservative surfaces of type $S_{C_1}$ corresponding to the same $C_1$ and symmetric each other, at the level of $C^\infty$ smoothness.

In fact, we will prove that we can glue two biconservative surfaces $S_{C_1}$ and $S_{C_1'}$, at the level of $C^\infty$ smoothness, only along the boundary. More precisely, let $S_{C_1'}$ given by
$$
X_{C_1'}(\rho,v)=\left(\rho\cos v+a_1\right)\overline{f}_1+\left(\rho\sin v+a_1\right)\overline{f}_2+\left(t_{C_1'}(\rho)+a_3\right)\overline{f}_3,
$$
where $\left\{\overline{f}_1, \overline{f}_2,\overline{f}_3\right\}$ is a positively oriented orthonormal basis of $\mathbb{R}^3$ and $a_1,a_2,a_3\in\mathbb{R}$. Assume that we can glue $\overline{S}_{C_1}$ and $S_{C_1'}$ along a curve $\gamma=\gamma(s)$, $\gamma'(s)\neq0$, for any $s$, at the level of $C^\infty$ smoothness. In this case we have
\begin{equation}
\label{sistem_gluing}
\left\{
\begin{array}{lll}
\gamma(s)\in\overline{S}_{C_1}\cap \overline{S}_{C_1'}\\
\eta_{C_1}(\gamma(s))\ ||\ \eta_{C_1'}(\gamma(s))\\
H_{C_1}(\gamma(s))=H_{C_1'}(\gamma(s))\\
\left(\grad\left|H_{C_1}\right|\right)\left(\gamma(s)\right)=(\grad|H_{C_1'}|)\left(\gamma(s)\right)
\end{array}
\right.,
\end{equation}
for any $s$, where the mean curvature vector field $H_{C_1}$ is given by $H_{C_1}=\frac{1}{2}f_{C_1}\eta_{C_1}$. For $S_{C_1}$ we have
\begin{eqnarray*}
\eta_{C_1}(\rho,v)&=&\frac{X_{C_1,\rho}\times X_{C_1,v}}{\left|X_{C_1,\rho}\times X_{C_1,v}\right|}\\
&=&-\frac{1}{\sqrt{C_1}\rho^{1/3}}\cos v\ \overline{e}_1-\frac{1}{\sqrt{C_1}\rho^{1/3}}\sin v\ \overline{e}_2+\sqrt{\frac{C_1\rho^{2/3}-1}{C_1\rho^{2/3}}}\ \overline{e}_3
\end{eqnarray*}
and the mean curvature function
\begin{eqnarray*}
f_{C_1}(\rho,v)&=&\left(1+\left(t_{C_1}'(\rho)\right)^2\right)^{-3/2}\left(t_{C_1}''(\rho)+\frac{t_{C_1}'(\rho)\left(1+\left(t_{C_1}'(\rho)\right)^2\right)}{\rho}\right)\\
&=&\frac{2}{3\sqrt{C_1}\rho^{4/3}}>0.
\end{eqnarray*}
It follows that $f_{C_1}(\rho,v)=f_{C_1}(\rho)$, $f_{C_1}=2\left|H_{C_1}\right|$, and
\begin{eqnarray*}
\left(\grad f_{C_1}\right)(\rho,v) &=&\frac{1}{1+\left(t_{C_1}'(\rho)\right)^2}\ f_{C_1}'(\rho)\ X_{C_1,\rho}(\rho,v)\\
&=&-\frac{8}{9C_1^{3/2}\rho^3}\left(\left(C_1\rho^{2/3}-1\right)\cos v\ \overline{e}_1+\left(C_1\rho^{2/3}-1\right)\sin v\ \overline{e}_2+\right.\\
&&\qquad\qquad\quad\left.+\sqrt{C_1\rho^{2/3}-1}\ \overline{e}_3\right).
\end{eqnarray*}
Similar formulas hold for $S_{C_1'}$. Now, let us consider
$$
\left(\rho_1(s),v_1(s)\right)=X_{C_1}^{-1}\circ \gamma(s)\  \text{ and}\ \left(\rho_2(s),v_2(s)\right)=X_{C_1'}^{-1}\circ \gamma(s).
$$
We can rewrite $(\ref{sistem_gluing})$ as
\begin{equation}
\label{sistem_gluing2}
\left\{
\begin{array}{lll}
X_{C_1}\left(\rho_1(s),v_1(s)\right)=X_{C_1'}\left(\rho_2(s),v_2(s)\right)\\
\eta_{C_1}\left(\rho_1(s),v_1(s)\right)= \eta_{C_1'}\left(\rho_2(s),v_2(s)\right)\\
f_{C_1}\left(\rho_1(s),v_1(s)\right)=f_{C_1'}\left(\rho_2(s),v_2(s)\right)\\
\left(\grad f_{C_1}\right)\left(\rho_1(s),v_1(s)\right)=(\grad f_{C_1'})\left(\rho_2(s),v_2(s)\right)
\end{array}
\right.,
\end{equation}
for any $s$, where $\rho_1(s)\geq C_1^{-3/2}$ and $\rho_2(s)\geq \left(C_1'\right)^{-3/2}$.

First, we can observe that $C_1\rho_1^{2/3}(s)-1=0$ if and only if $C_1'\rho_2^{2/3}(s)-1=0$. Next, we consider two cases.

In the first case, when $C_1\rho_1^{2/3}(s)-1=0$ for any $s$, by a straightforward computation, from the third relation of $(\ref{sistem_gluing2})$, we can see that $C_1=C_1'$ and $\rho_1(s)=\rho_2(s)=C_1^{-3/2}$, for any $s$. Moreover, $t_{C_1}\left(\rho_1(s)\right)=0$ and $t_{C_1'}\left(\rho_2(s)\right)=0$. Then, from the first relation we get $a_1=a_2=a_3=0$ and $\langle\overline{e}_1, \overline{f}_3\rangle=\langle\overline{e}_2, \overline{f}_3\rangle=0$, i.e., $\overline{e}_3=\pm\overline{f}_3$.
Therefore, $S_{C_1}$ and $S_{C_1'}$ coincide or one of them is the symmetric of another with respect to the affine plane where the common boundary lies.

In the second case, we suppose that there exists $s_0$ such that $C_1\rho_1^{2/3}(s_0)-1\neq0$. It follows that also $C_1'\rho_2^{2/3}(s_0)-1\neq0$. Thus, we get that $C_1\rho_1^{2/3}(s)-1>0$ and $C_1'\rho_2^{2/3}(s)-1>0$ around $s_0$. By direct computation, from $(\ref{sistem_gluing2})$, we obtain $C_1=C_1'$, $a_1=a_2=a_3=0$, $\rho_1(s)=\rho_2(s)$ around $s_0$, and $\langle\overline{e}_3,\overline{f}_3\rangle=1$, i.e., $\overline{e}_3=\overline{f}_3$. Therefore, in this case $S_{C_1}$ and $S_{C_1'}$ coincide.

However, we must then check that we have a smooth gluing.

\end{proof}

\begin{proposition}[\cite{MOR3}]
Any two complete biconservative surfaces differ by an homothety of $\mathbb{R}^3$.
\end{proposition}

\begin{proof}
First, let us consider a reparameterization of the profile curve (we consider only the upper part)
$$
\sigma_{C_1}(\rho)=(\rho,0,t_{C_1}(\rho))\equiv (\rho,t_{C_1}(\rho)), \qquad \rho > C_{1}^{-3/2},
$$
by considering the change of coordinate $\theta=C_1\rho^{2/3}-1$, $\theta>0$. Then we get
$$
\sigma_{C_1}(\theta)=\left(\sigma^1_{C_1}(\theta),\sigma^2_{C_1}(\theta)\right)=C_1^{-3/2}\left((\theta+1)^{3/2},\frac{3}{2}\left[\sqrt{\theta^2+\theta}+\log(\sqrt{\theta}+\sqrt{\theta+1})\right]\right),
$$
where $\theta>0$, and
\begin{eqnarray*}
X_{C_1}(\theta,v)&=&C_1^{-3/2}\left( (\theta+1)^{3/2}\cos v,(\theta+1)^{3/2} \sin v,\right.\\
&&\left.\frac{3}{2}\left[\sqrt{\theta^2+\theta}+\log(\sqrt{\theta}+\sqrt{\theta+1})\right]\right),
\end{eqnarray*}
for $\theta>0$ and $v\in\mathbb{R}$, i.e.,
$$
X_{C_1}(\theta,v)=C_1^{-3/2}X_1(\theta,v), \qquad \theta>0, v\in\mathbb{R}.
$$
Thus we get $\tilde{S}_{C_1}=C_1^{-3/2}\tilde{S}_1$.
\end{proof}

\subsection{Biconservative surfaces in $\mathbb{S}^3$}

$\newline$
The local classification of biconservative surfaces with $|\grad f|>0$ in the 3-dimensional unit Euclidean sphere is given by the following result.

\begin{theorem}[\cite{CMOP}]
Let $M^2$ be a biconservative surface in $\mathbb{S}^3$ with $f(p)>0$ and $(\grad f)(p)\neq 0$ at any point $p\in M$. Then, locally, $M^2\subset \mathbb{R}^4$ can be parameterized by
\begin{equation}
\label{Phi_{C_1}}
\Phi_{C_1}(u,v)=\sigma(u)+\frac{4}{3\sqrt{C_1}k(u)^{3/4}}\left(\overline{f}_{1}(\cos v -1)+\overline{f}_{2}\sin v\right),
\end{equation}
where $C_1$ is a positive constant; $\overline{f}_{1},\overline{f}_{2}\in \mathbb{R}^4$ are two orthonormal constant vectors; $\sigma(u)$ is a curve parameterized by arclength that satisfies
\begin{equation}
\label{eq:sigma_prod_scal}
\langle \sigma(u),
\overline{f}_{1}\rangle=\frac{4}{3\sqrt{C_1}k(u)^{3/4}}, \qquad \langle \sigma(u),\overline{f}_{2}\rangle=0,
\end{equation}
and whose curvature $k=k(u)$ is a positive non-constant solution of the following ODE
\begin{equation}
\label{k''k}
k''k=\frac{7}{4}(k')^2+\frac{4}{3}k^2-4k^4.
\end{equation}
\end{theorem}

\begin{remark}
The curve $\sigma$ lies in the totally geodesic $\mathbb{S}^2=\mathbb{S}^3\cap\Pi$, where $\Pi$ is the linear hyperspace of $\mathbb{R}^4$ orthogonal to $\overline{f}_{2}$.
\end{remark}

In the following, we will prove that such a curve $\sigma$ exists and will find a more explicit expression for $(\ref{Phi_{C_1}})$.

First, we observe that $(\ref{k''k})$ has the prime integral
\begin{equation}
\label{eq:prime_integ}
\left(k'\right)^2=-\frac{16}{9}k^2-16k^4+C_1k^{7/2}.
\end{equation}
Replacing $(\ref{eq:prime_integ})$ in $(\ref{k''k})$, since $k'\neq0$, we get
\begin{equation*}
k''=-\frac{16}{9}k-32k^3+\frac{7}{4}C_1k^{5/2}.
\end{equation*}

In order to prove the existence of such a curve $\sigma$, we will follow a slightly different method from that in \cite{CMOP}. We consider $\overline{f}_{1}=\overline{e}_3$ and $\overline{f}_{2}=\overline{e}_4$, where $\left\{\overline{e}_1,\overline{e}_2,\overline{e}_3,\overline{e}_4\right\}$ is the canonical basis of $\mathbb{R}^4$.

From $(\ref{eq:sigma_prod_scal})$ it follows that $\sigma$ can be written as
$$
\sigma(u)=\left(x(u),y(u),\frac{4}{3\sqrt{C_1}}k(u)^{-3/4},0\right).
$$
Using polar coordinates, we have $x(u)=R(u)\cos \mu(u)$ and $y(u)=R(u)\sin \mu(u)$, with $R(u)>0$.

Since $\sigma(u)\subset \mathbb{S}^3$, $R^2=x^2+y^2$ and $R>0$, we get $k>\left(\frac{16}{9C_1}\right)^{2/3}$ and
\begin{equation}
\label{eq:R}
R=\sqrt{1-\frac{16}{9C_1}k^{-3/2}}.
\end{equation}

As $k'(u)\neq 0$, we can view $u$ as a function of $k$, and considering $R=R(u(k))$ and $\mu=\mu(u(k))$, by a straightforward computation, it follows that $\sigma$ is explicitly given by
$$
\sigma(k)=\left(R\cos\mu,R\sin\mu,\frac{4}{3\sqrt{C_1}}k^{-3/4},0\right),
$$
where $R$ is given by $(\ref{eq:R})$ and
$$
\mu(k)=\pm \bigintsss_{k_0}^k {\frac{108\sqrt{\frac{\tau^2} {-16+9C_1\tau^{3/2}}}}{\sqrt{\tau^{1/2}\frac{\left(-16+9C_1\tau^{3/2}\right)\left(9C_1\tau^{3/2}-16\left(1+9\tau^2\right)\right)}{C_1}}}d\tau}+c_0,
$$
where $c_0$ is a real constant.

If we use the formula of $\sigma$ in $(\ref{Phi_{C_1}})$, we get
\begin{equation*}
\Phi_{C_1}(k,v)=\left(\sqrt{1-\frac{16}{9C_1}k^{-3/2}}\cos \mu, \sqrt{1-\frac{16}{9C_1}k^{-3/2}}\sin \mu, \frac{4\cos v}{3\sqrt{C_1}k^{3/4}},\frac{4\sin v}{3\sqrt{C_1}k^{3/4}}\right).
\end{equation*}

Next, we have to determine the maximum domain for $\Phi_{C_1}$. From $(\ref{eq:prime_integ})$, we ask that $-\frac{16}{9}k^2-16k^4+C_1k^{7/2}>0$. Since $k>0$, it is enough to find the interval where $-\frac{16}{9}-16k^2+C_1k^{3/2}>0$. We denote by
$$
L(k)=-\frac{16}{9}-16k^2+C_1k^{3/2},\qquad k>0.
$$
We can see that if $C_1>\frac{64}{3^{5/4}}$, one obtains that there exist exactly two $k_{01}\in \left(0,\left(\frac{3}{64}C_1\right)^{2}\right)$ and  $k_{02}\in \left(\left(\frac{3}{64}C_1\right)^{2}, \infty\right)$ such that $L(k_{01})=L(k_{02})=0$ and $L(k)>0$ for any $k\in \left(k_{01},k_{02}\right)$.

We note that $k_{01}>\left(\frac{16}{9C_1}\right)^{2/3}$.

Therefore, the domain of $\Phi_{C_1}$ is $\left(k_{01},k_{02}\right)\times\mathbb{R}$, where $k_{01}$ and $k_{02}$ are the vanishing points of $F$, with $0<k_{01}<k_{02}$.

\begin{remark}
We can choose $c_0=0$ in the above expression of $\mu$, by considering a linear orthogonal transformation of $\mathbb{R}^4$.
\end{remark}

We end this section, by recalling the following result from \cite{FNO}, where the necessary and sufficient conditions for an abstract Riemannian surface to admit a biconservative immersion in $N^3(c)$ were determined.

\begin{theorem}[\cite{FNO}]
\label{thm:char}
Let $(M^2,g)$ be a Riemannian surface and $c\in\mathbb{R}$ a real constant. Then $M$ can be locally isometrically embedded in a space form $N^3(c)$ as a biconservative surface with positive mean curvature having the gradient different from zero at any point $p\in M$ if and only if the Gaussian curvature $K$ satisfies $c-K(p)>0$, $(\grad K)(p)\neq 0$, and its level curves are circles in $M$ with curvature $\kappa=(3|\grad K|)/(8(c-K))$.
\end{theorem}

\begin{remark}[\cite{FNO}]
The level curves of $K$ are circles with constant curvature
$$
\kappa=\frac{3|\grad K|}{8(c-K)}
$$
if and only if $X_2X_1 K=0$ and $\nabla_{X_2}X_2=-\frac{3X_1 K}{8(c-K)}X_1$, where $X_1=\frac{\grad K}{|\grad K|}$ and $X_2\in C(TM)$ are two vector fields on $M$ such that $\{X_1(p),X_2(p)\}$ is a positively oriented orthonormal basis at any point $p\in M$.
\end{remark}

\begin{remark}[\cite{FNO}]
In the case of biconservative immersions, we have a rigidity result. Indeed, let $(M^2,g)$ be a simply connected Riemannian surface and $c\in\mathbb{R}$ a constant. If $M$ admits two biconservative Riemannian immersions in $N^3(c)$ such that their mean curvatures are positive with gradients different from zero at any point $p\in M$, then the two immersions differ by an isometry of $N^3(c)$.
\end{remark}

\section{Intrinsic characterisation of biconservative surfaces in $\mathbb{R}^3$ and $\mathbb{S}^3$}\label{icbs}

In \cite{FNO}, the metric of an abstract Riemannian surface $\left(M^2,g\right)$ which admits a biconservative immersion with $|\grad f|>0$ in $N^3(c)$ was not explicitly determined. We will find this metric in an explicit way.

First, we have the next proposition.

\begin{proposition}
\label{prop_3.1}
Let $(M^2,g)$ be a Riemannian surface with Gaussian curvature $K$ satisfying $(\grad K)(p)\neq 0$ and $c-K(p)>0$ at any point $p\in M$, where $c\in \mathbb{R}$ is a constant. Let $X_1=\frac{\grad K}{|\grad K|}$ and $X_2\in C(TM)$ be two vector fields on $M$ such that $\{X_1(p),X_2(p)\}$ is a positively oriented orthonormal basis at any point $p\in M$. Then $X_2X_1 K=0$ and $\nabla_{X_2}X_2=-\frac{3X_1 K}{8(c-K)}X_1$ if and only if the Riemannian metric $g$ can be locally written as $g=e^{2\varphi(u)}(du^2+dv^2)$, where $\varphi$ satisfies the equation
$$
8c e^{2\varphi(u)}\varphi'(u) +2\varphi'(u) \varphi''(u)+3\varphi'''(u)=0
$$
and the conditions
$$
K'(u)=e^{-2\varphi(u)}(2\varphi'(u)\varphi''(u)-\varphi'''(u))\neq 0
$$ and
$$
c-K(u)=c+e^{-2\varphi(u)}\varphi''(u)>0,
$$
for any $u$ in some open interval $I$.
\end{proposition}

\begin{proof}
In \cite{FNO} we have seen that if we have a Riemannian surface with Gaussian curvature $K$ satisfying $(\grad K)(p)\neq 0$ and $c-K(p)>0$ at any point $p\in M$, where $c\in \mathbb{R}$ is a constant, $X_1=\frac{\grad K}{|\grad K|}$ and $X_2\in C(TM)$ are two vector fields on $M$ such that $\{X_1(p),X_2(p)\}$ is a positively oriented orthonormal basis at any point $p\in M$ such that $X_2X_1 K=0$ and $\nabla_{X_2}X_2=-\frac{3X_1 K}{8(c-K)}X_1$, then the Riemannian metric $g$ can be locally written as
$$
g=e^{2\varphi(u)}(du^2+dv^2),
$$
where $(W;u,v)$ is a positive isothermal chart.

Let $p_0$ be a fixed point in $M$ and $X=X(u,v)$ be a local parametrization of $M$ in a neighborhood $U\subset M$ of $p_0$, positively oriented.

Identifying $K$ with $K\circ X $ we get the following formulas. The Gaussian curvature is given by $K(u)=-e^{-2\varphi(u)}\varphi''(u)$, $(\grad K)(u)=e^{-2\varphi(u)}K'(u)\partial_u$ and $|\grad K|=e^{-\varphi(u)}|K'(u)|$. By hypothesis, we have that $c-K(u)>0$, and therefore
\begin{equation*}
c + e^{-2\varphi(u)}\varphi''(u)>0,
\end{equation*}
for any $u$.

Since $\grad K\neq 0$ at any point of $M$, we can assume that $K'(u)>0$ for any $u$. Then it follows that $X_1=e^{-\varphi(u)}\partial_u$ and $X_2=e^{-\varphi(u)}\partial_v$. It is easy to see that $\nabla_{X_2}X_2=-e^{-2\varphi}\varphi'(u)\partial_u$. Thus $\nabla_{X_2}X_2=-\frac{3X_1 K}{8(c-K)}X_1$ if and only if
\begin{equation*}
-e^{-2\varphi}\varphi'(u)\partial_u = -\frac{3e^{-4\varphi(u)}\left(2\varphi'(u)\varphi''(u)-\varphi'''(u)\right)}{8\left(c+e^{-2\varphi(u)}\varphi''(u)\right)}\partial_u,
\end{equation*}
which means that
\begin{equation}
\label{eq:ode_prop}
8c e^{2\varphi(u)}\varphi'(u) +2\varphi'(u) \varphi''(u)+3\varphi'''(u)=0.
\end{equation}

The converse is immediate.
\end{proof}

\begin{remark}
In Proposition $\ref{prop_3.1}$, if we assume that $K'(u)<0$ for any $u$, we obtain the same ODE for $\varphi$ to satisfy.
\end{remark}

Applying the above result to the case $c=0$ we get our next result.
\begin{proposition}
\label{prop4.3}
Let $\left(M^2,g\right)$ be a Riemannian surface with Gaussian curvature $K$ satisfying $(\grad K)(p)\neq 0$ and $K(p)<0$ at any point $p\in M$. Let $X_1=\frac{\grad K}{|\grad K|}$ and $X_2\in C(TM)$ be two vector fields on $M$ such that $\{X_1(p),X_2(p)\}$ is a positively oriented orthonormal basis at any point $p\in M$. Then $X_2X_1 K=0$ and $\nabla_{X_2}X_2=\frac{3X_1 K}{8K}X_1$ if and only if the Riemannian metric $g$ can be locally written as
$$
g_C(u,v)=C\left(\cosh u\right)^6(du^2+dv^2), \qquad u\neq0,
$$
where $C\in \mathbb{R}$ is a positive constant.
\end{proposition}

\begin{proof}
For $c=0$, equation (\ref{eq:ode_prop}) becomes

\begin{equation}
\label{eq:ode_cor1}
3\varphi'''(u)+2\varphi'(u)\varphi''(u)=0.
\end{equation}
Since $K=-e^{-2\varphi(u)}\varphi''(u)<0$, we obtain $\varphi''(u)>0$ for any $u$.

By a straightforward computation, we get the unique solution of (\ref{eq:ode_cor1})
\begin{equation}
\label{good_solution}
\varphi(u)=a\int_{u_0'}^u{\frac{1-e^{-\frac{2a}{3}\left(\tau+u_0\right)}}{1+e^{-\frac{2a}{3}\left(\tau+u_0\right)}}d\tau} +b_1, \qquad u\in I,
\end{equation}
where $a,b_1,u_0\in \mathbb{R}$, $I$ is an open interval and $u_0'\in I$ is arbitrary fixed.

Next, we will compute the integral in $(\ref{good_solution})$, also imposing $K'(u)>0$. First, we will show that $K'(u)>0$ if and only if $u+u_0>0$.

Since
\begin{equation}
\label{eq:K}
K(u)=-e^{-2\varphi(u)}\varphi''(u), \qquad u\in I,
\end{equation}
we have that
\begin{equation*}
K'(u)=e^{-2\varphi(u)}\left(2\varphi'(u)\varphi''(u)-\varphi'''(u)\right)>0, \qquad u\in I,
\end{equation*}
if and only if
\begin{equation}
\label{eq:K'}
2\varphi'(u)\varphi''(u)-\varphi'''(u)>0,\qquad u\in I.
\end{equation}
From $(\ref{good_solution})$ we get
$$
\varphi'''(u)=-\frac{8a^3 e^{-\frac{2a}{3}(u+u_0)}\left(1-e^{-\frac{2a}{3}\left(u+u_0\right)}\right)}{9\left(1+e^{-\frac{2a}{3}(u+u_0)}\right)^3}.
$$
If we replace the first, the second and the third derivatives of $\varphi$ in $(\ref{eq:K'})$, we obtain that $K'(u)>0$ if and only if $a^3\left(1-e^{-\frac{2a}{3}\left(u+u_0\right)}\right)>0$. It is easy to check that this is equivalent to $u+u_0>0$ if either $a>0$ or $a<0$.

Therefore, the solution is
$$
\varphi(u)=a\int_{u_0'}^u{\frac{1-e^{-\frac{2a}{3}(\tau+u_0)}}{1+e^{-\frac{2a}{3}(\tau+u_0)}}d\tau} +b_1, \qquad u\in I,u+u_0>0,
$$
where $b_1,u_0\in \mathbb{R}, a\in \mathbb{R}^\ast$, $I$ is an open interval and $u_0'\in I$ is arbitrary fixed.

Then, in order to compute the integral in $(\ref{good_solution})$, we consider some changes of variables and obtain
$$
\varphi(u)=3 \log\left(\cosh \frac{u}{3}\right)+b_2, \qquad u\in I,u>0,
$$
where $b_2\in \mathbb{R}$ and $I$ is an open interval.

If we impose $K'(u)<0$, then from $(\ref{good_solution})$, following the same steps as above, we obtain
$$
\varphi(u)=3 \log\left(\cosh \frac{u}{3}\right)+b_2, \qquad u\in I,u<0,
$$
where $b_2\in \mathbb{R}$ and $I$ is an open interval.

Since $g=e^{2\varphi(u)}\left(du^2+dv^2\right)$, by a new change of coordinates, we come to the conclusion, i.e.,
$$
g_C=C \left(\cosh u\right)^6\left(du^2+dv^2\right),
$$
where $(W;u,v)$ is a positive isothermal chart, $u\neq 0$, and $C\in \mathbb{R}$ is a positive constant.
\end{proof}

Using Proposition $\ref{prop_3.1}$ in the case when $c=1$, we obtain the following result.
\begin{proposition}
\label{prop4.5}
Let $(M^2,g)$ be a Riemannian surface with Gaussian curvature $K$ satisfying $(\grad K)(p)\neq 0$ and $1-K(p)>0$ at any point $p\in M$. Let $X_1=\frac{\grad K}{|\grad K|}$ and $X_2\in C(TM)$ be two vector fields on $M$ such that $\{X_1(p),X_2(p)\}$ is a positively oriented orthonormal basis at any point $p\in M$. Then $X_2X_1 K=0$ and $\nabla_{X_2}X_2=-\frac{3X_1 K}{8(1-K)}X_1$ if and only if the Riemannian metric $g$ can be locally written as $g=e^{2\varphi(u)}(du^2+dv^2)$, where $u=u(\varphi)$ satisfies

$$
u=u(\varphi)=\pm \int_{\varphi_0}^\varphi \frac{d\tau}{\sqrt{\frac{b}{3}e^{-\frac{2}{3}\tau}-e^{2\tau}+a}}+c, \qquad \varphi \in I,
$$
where $a,b,c\in \mathbb{R}$, $a>0$, $b<0$, and $\frac{b}{3}e^{-\frac{2}{3}\varphi}-e^{2\varphi}+a>0$ for every $\varphi \in I$, where $I$ is some open interval.
\end{proposition}

\begin{proof}
When $c=1$, equation (\ref{eq:ode_prop}) becomes
\begin{equation}
\label{eq:ode_cor2}
3\varphi'''(u)+2\varphi'(u)\varphi''(u)+8e^{2\varphi(u)}\varphi'(u)=0.
\end{equation}

We note that $(\ref{eq:ode_cor2})$ can be written as $\left(3\varphi''+(\varphi')^2+4e^{2\varphi}\right)'(u)=0$ and, integrating, we obtain the prime integral of $(\ref{eq:ode_cor2})$
\begin{equation*}
3\varphi''(u)+\left(\varphi'(u)\right)^2+4e^{2\varphi(u)}=a,
\end{equation*}
where $a\in \mathbb{R}$ is a constant. From this equation we have that
\begin{equation}
\label{eq:c=1}
e^{-2\varphi(u)}\varphi''(u)=\frac{1}{3}a e^{-2\varphi(u)}-\frac{1}{3}e^{-2\varphi(u)}\left(\varphi'(u)\right)^2-\frac{4}{3}.
\end{equation}
Since $K(u)=-e^{-2\varphi(u)}\varphi''(u)$, from $(\ref{eq:c=1})$, we obtain that $1-K(u)>0$ for any $u$ if and only if $e^{-2\varphi(u)}\left(a-\left(\varphi'(u)\right)^2\right)>1$.

It is easy to see that $a$ has to be grater than $(\varphi'(u))^2$, so that $a$ is a positive real number.

We note that, if $\varphi'=0$, then $K=0$ and $\grad K=0$, which contradicts the hypotheses. Therefore, we will assume that $\varphi'\neq 0$.

As $\varphi'(u)\neq 0$, we can view $u$ as a function of $\varphi$ and, by direct computation we get
$$
u(\varphi)=\pm \int_{\varphi_0}^\varphi {\frac{d\tau}{\sqrt{\frac{b}{3}e^{-\frac{2}{3}\tau}-e^{2\tau}+a}}}+c, \qquad \varphi\in I,
$$
where $a,b,c\in \mathbb{R}$, $a>0$, $b<0$, and $\frac{b}{3}e^{-\frac{2}{3}\varphi}-e^{2\varphi}+a>0$, for every $\varphi \in I$, where $I$ is some open interval.
\end{proof}

We note that in Proposition $\ref{prop4.5}$, if $K'>0$, then
$$
u(\varphi)=\int_{\varphi_0}^\varphi {\frac{d\tau}{\sqrt{\frac{b}{3}e^{-\frac{2}{3}\tau}-e^{2\tau}+a}}}+c, \qquad \varphi\in I,
$$
and, if $K'<0$, then
$$
u(\varphi)=- \int_{\varphi_0}^\varphi {\frac{d\tau}{\sqrt{\frac{b}{3}e^{-\frac{2}{3}\tau}-e^{2\tau}+a}}}+c, \qquad \varphi\in I.
$$

\begin{remark}
A similar result to Proposition $\ref{prop4.5}$ can be obtained when $c=-1$.
\end{remark}

\section{Global properties of biconservative surfaces in $\mathbb{R}^3$ and $\mathbb{S}^3$}
In the previous section we determined (locally) all abstract Riemannian surfaces which admit biconservative immersions with $\grad f\neq 0$ in $\mathbb{R}^3$ or $\mathbb{S}^3$  (and we know that such an immersion is unique). Next, we will find the explicit expressions of complete biconservative surfaces in $\mathbb{R}^3$ and $\mathbb{S}^3$.

\subsection{Biconservative surfaces in $\mathbb{R}^3$}

$\newline$
In the case of complete biconservative surfaces in $\mathbb{R}^3$, we have the following result

\begin{theorem}
\label{main_th1}
Let $\left(\mathbb{R}^2,g_C=C \left(\cosh u\right)^6\left(du^2+dv^2\right)\right)$ be a Riemannian surface, where $C$ is a positive constant. Then we have:
\begin{itemize}
\item[(a)] the metric on $\mathbb{R}^2$ is complete;
\item[(b)] $K_C(u,v)=K_C(u)=-\frac{3}{C\left(\cosh u\right)^8}<0$, $K'_C(u)=\frac{24}{C}\frac{\sinh u}{\left(\cosh u\right)^9}$, and therefore $\grad K_C\neq 0$ at any point of $\mathbb{R}^2\setminus Ov$;

\item[(c)] the immersion $X_C:\left(\mathbb{R}^2,g_C\right)\to \mathbb{R}^3$ given by
    $$
    X_C(u,v)=\left(\sigma_C^1(u)\cos 3v, \sigma_C^1(u)\sin 3v, \sigma_C^2(u)\right)
    $$
    is biconservative in $\mathbb{R}^3$, where
    $$
    \sigma_C^1(u)=\frac{C^{1/2}}{3}\left(\cosh u\right)^3, \quad
    \sigma_C^2(u)=\frac{C^{1/2}}{2}\left(\frac{1}{2}\sinh {2u}+u\right), \qquad u \in \mathbb{R}.
    $$
\end{itemize}
\end{theorem}

\begin{proof}
In order to prove $(a)$, we will use Proposition $\ref{prop1}$ .

Consider $g_0=du^2+dv^2$ the Euclidian  metric on $\mathbb{R}^2$, which is complete. Then, denote by $\tilde{g}$ the Riemannian metric $\tilde{g}=(\cosh u)^6 g_0$, and note that $\tilde{g}-g_0=\left(\left(\cosh u\right)^6-1\right)g_0$ is non-negative definite at any point of $\mathbb{R}^2$. Therefore $\tilde{g}$ is also complete and since $g_C=C\tilde{g}$, it follows that $\left(\mathbb{R}^2,g_C\right)$ is complete.

To prove $(b)$, we consider the formula $(\ref{eq:K})$, with $\varphi(u)=\log\left(\sqrt{C}\left(\cosh u\right)^3\right)$ and obtain that the Gaussian curvature $K_C(u,v)$ is equal to
$$
K_C(u,v)=K_C(u)=-\frac{3}{C\left(\cosh u\right)^8}
$$
and
$$
K_C'(u)=\frac{24}{C}\frac{\sinh u}{\left(\cosh u\right)^9}.
$$

Therefore, $K_C'(u)>0$ if and only if $u>0$, $K_C'(u)<0$ if and only if $u<0$, and $K_C'(0)=0$. Since
$$
\left(\grad K_C\right)(u,v)=\frac{1}{C}e^{-6\log\left(\cosh u\right)}K_C'(u)\partial_u,
$$
we have $\grad K_C\neq 0$ at any point of $\mathbb{R}^2\setminus Ov$, which is an open dense subset of $\mathbb{R}^2$.

We begin the proof of $(c)$, recalling that if we have a biconservative surface of revolution in $\mathbb{R}^3$, with non-constant mean curvature, its profile curve is
\begin{eqnarray*}
\sigma_{C_1}^+(\theta)&=&\left(\sigma_{C_1}^1(\theta),\sigma_{C_1}^2(\theta)\right)\\
&=& C_1^{-3/2}\left((\theta+1)^{3/2},\frac{3}{2}\left[\sqrt{\theta^2+\theta}+\log(\sqrt{\theta}+\sqrt{\theta+1})\right]\right), \qquad \theta>0,
\end{eqnarray*}
and
\begin{eqnarray*}
X_{C_1}^+(\theta,v)&=&C_1^{-3/2}\left( (\theta+1)^{3/2}\cos v,(\theta+1)^{3/2} \sin v,\right.\\
&&\qquad\quad\left.\frac{3}{2}\left[\sqrt{\theta^2+\theta}+\log(\sqrt{\theta}+\sqrt{\theta+1})\right]\right), \qquad \theta>0, v\in\mathbb{R}.
\end{eqnarray*}

To compute the metric on this surface we first need the coefficients of the first fundamental form
\begin{equation*}
E_{C_1}^+(\theta,v)=\frac{1}{C_1^3}\frac{9(\theta+1)^2}{4\theta}, \qquad F_{C_1}^+(\theta,v)=0, \qquad G_{C_1}^+(\theta,v)=\frac{1}{C_1^3}(\theta+1)^3.
\end{equation*}
Thus, the Riemannian metric is
$$
g_{C_1}^+(\theta,v)=\frac{1}{C_1^3}\left(\frac{9(\theta+1)^2}{4\theta}d\theta^2+(\theta+1)^3 dv^2\right).
$$
If we consider the change of coordinates $\left(\theta,v\right)=\left(\left(\sinh u\right)^2,3v\right)$, where $u\neq 0$, we obtain
$$
g_{C_1}^+(u,v)=\frac{9}{C_1^3}\left(\cosh u\right)^6 \left(du^2+dv^2\right).
$$
Since $C_1$ is an arbitrary positive constant, we can consider $C_1=\left(\frac{9}{C}\right)^{1/3}$, where $C$ is the positive constant corresponding to $g_C$, and therefore $g_{C_1}^+=g_C$.

Now, we can find a biconservative immersion from the half plane $u>0$ with the metric $g_C$ in $\mathbb{R}^3$ . The profile curve can now be written as
{\small{
\begin{eqnarray*}
\sigma_{C}^+(u)&=&\left(\sigma_{\left(\frac{9}{C}\right)^{1/3}}^1\left(u\right),\sigma_{\left(\frac{9}{C}\right)^{1/3}}^2\left(u\right)\right)\\
&=& \frac{C^{1/2}}{3}\left(\left(\cosh u\right)^3,\frac{3}{2}\left(\sinh u\cosh u+\log\left(\sinh u+\cosh u\right)\right)\right)\\
&=& \frac{C^{1/2}}{3}\left(\left(\cosh u\right)^3,\frac{3}{2}\left(\frac{1}{2}\sinh 2u+u\right)\right), \qquad u>0.
\end{eqnarray*}
}}
Therefore, the biconservative immersion from the half plane $u>0$ with the metric $g_C$ in $\mathbb{R}^3$ is given by
$$
X_{C}^+(u,v)=\frac{C^{1/2}}{3}\left(\left(\cosh u\right)^3 \cos 3v, \left(\cosh u\right)^3 \sin 3v,\frac{3}{2}\left(\frac{1}{2}\sinh 2u+u\right)\right),
$$
where $u>0$ and $v\in \mathbb{R}$.

For the other half plane, i.e., $u<0$, using the symmetry with respect to $O\rho$, we define the profile curve
{\small{
\begin{eqnarray*}
\sigma_{C}^-(u)&=&\left(\sigma_{\left(\frac{9}{C}\right)^{1/3}}^1\left(-u\right),-\sigma_{\left(\frac{9}{C}\right)^{1/3}}^2\left(-u\right)\right)\\
&=& \frac{C^{1/2}}{3}\left(\left(\cosh u\right)^3,\frac{3}{2}\left(\frac{1}{2}\sinh 2u+u\right)\right),\qquad u<0.
\end{eqnarray*}
}}

Now, it is easy to see that we have a biconservative immersion, in fact a biconservative embedding from the whole $\left(\mathbb{R}^2,g_{C}\right)$ in $\mathbb{R}^3$, given by
$$
X_{C}(u,v)=\frac{C^{1/2}}{3}\left(\left(\cosh u\right)^3 \cos 3v, \left(\cosh u\right)^3 \sin 3v,\frac{3}{2}\left(\frac{1}{2}\sinh 2u+u\right)\right).
$$
\end{proof}

\begin{remark}
For $C=1$ the plot of the profile curve of $X_1$ is

\begin{center}
\begin{tikzpicture}[xscale=1,yscale=1]

\node[inner sep=0pt]  at (3,0){\includegraphics[width=.4\textwidth]{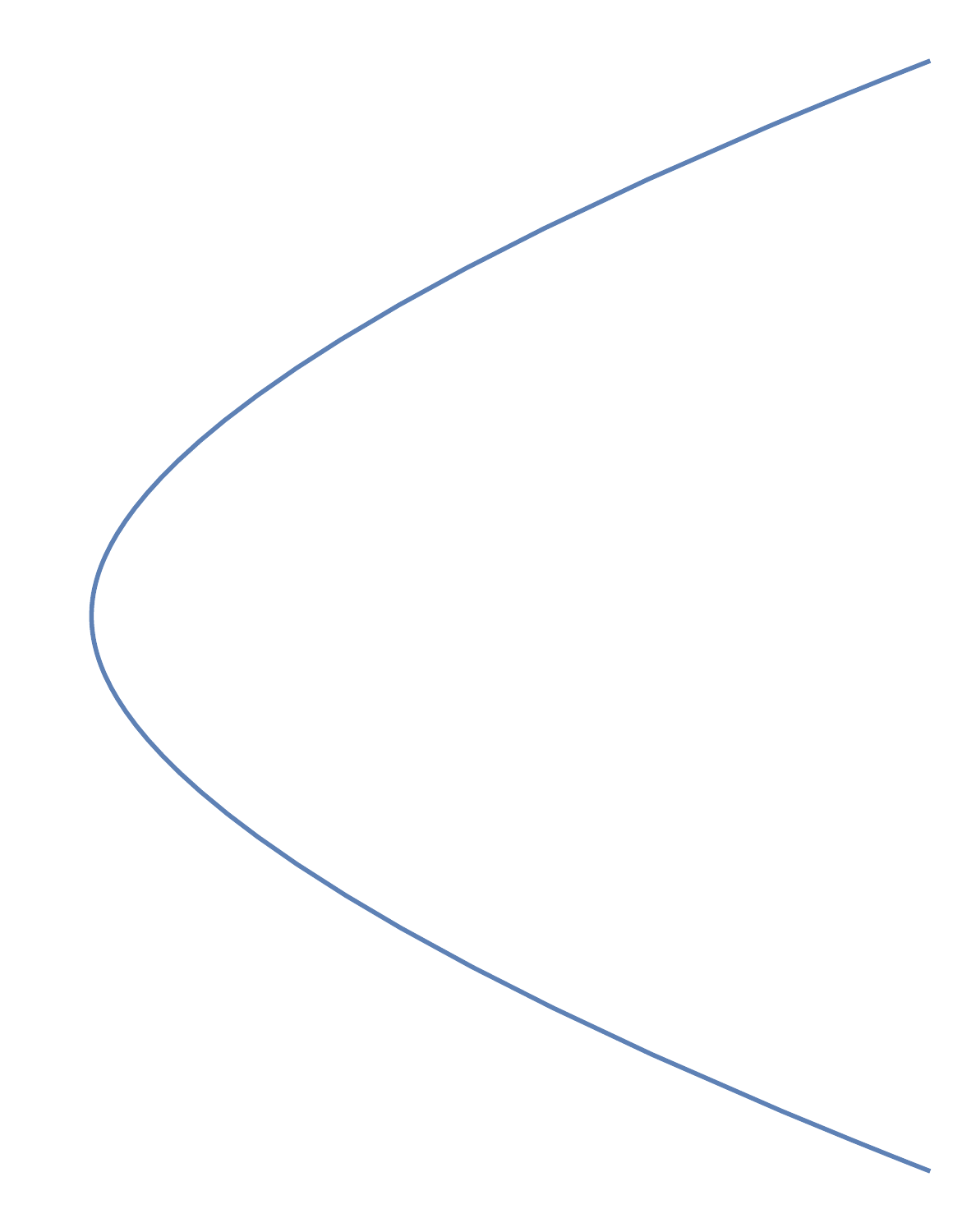}};
\draw [->,>=triangle 45] (-.5,0) -- (6,0);
\draw [->,>=triangle 45] (0,-4) -- (0,4);
\draw (6,0) node[below] {$x$};
\draw (0,4) node[left] {$z$};
\end{tikzpicture}

Figure 1. Plot of the profile curve $\left(\sigma_1^1(u),\sigma_1^2(u)\right)$.
\end{center}

\end{remark}

\subsection{Biconservative surfaces in $\mathbb{S}^3$}
$\newline$
Finding the explicit expressions of complete biconservative surfaces in $\mathbb{S}^3$ is more complicated and we will need some intermediate results.

\begin{proposition}
\label{prop:graf_F}
Let $\left(M^2,g\right)$ be a Riemannian surface with $g=e^{2\varphi(u)}(du^2+dv^2)$, where $u=u(\varphi)$ satisfies
$$
u=u(\varphi)=\pm \int_{\varphi_0}^\varphi \frac{d\tau}{\sqrt{\frac{b}{3}e^{-2\tau/3}-e^{2\tau}+a}}+c, \qquad \varphi\in I,
$$
where $a,b,c\in \mathbb{R}$, $a>0$, $b<0$, and $\frac{b}{3}e^{-2\varphi/3}-e^{2\varphi}+a>0$ for every $\varphi \in I$, with $I$ some open interval. Then $\left(M^2,g\right)$ is isometric to
$$
\left(D_C, g_C=\frac{3}{\xi^2(-\xi^{8/3}+3C\xi^2-3)}d\xi^2+\frac{1}{\xi^2}d\theta^2\right),
$$
where  $D_C=\left(\xi_{01},\xi_{02}\right)\times\mathbb{R}$, $C\in \left(\frac{4}{3^{3/2}},\infty\right)$ is a positive constant,  and $\xi_{01}$ and $\xi_{02}$ are the positive vanishing points of $-\xi^{8/3}+3C\xi^2-3$, with $0<\xi_{01}<\xi_{02}$.
\end{proposition}

\begin{proof}
Since
$$
u=u(\varphi)=\pm \int_{\varphi_0}^\varphi \frac{d\tau}{\sqrt{\frac{b}{3}e^{-2\tau/3}-e^{2\tau}+a}}+c,
$$
we have that
$$
du=\pm \frac{1}{\sqrt{\frac{b}{3}e^{-e^{-2\varphi /3}}-e^{2\varphi}+a}}d\varphi,
$$
and the expression of metric $g(u,v)=e^{2\varphi(u)}(du^2+dv^2)$ can be rewritten as
$$
g(\varphi,v)=\frac{e^{2\varphi}}{\frac{b}{3}e^{-e^{-2\varphi /3}}-e^{2\varphi}+a}d\varphi^2+e^{2\varphi}dv^2.
$$
Consider the change of coordinates $(\varphi,v)=\left(\log\left(\frac{(-b)^{3/8}}{\xi}\right),v\right)$ and we get that
$$
g(\xi,v)=\frac{1}{\xi^2}\left( \frac{3}{-\xi^{8/3}+3a(-b)^{-3/4}\xi^2-3}d\xi^2+(-b)^{3/4}dv^2\right).
$$
Now, considering another change of coordinates $(\xi,v)=\left(\xi,(-b)^{-3/8}\theta\right)$ and denoting $C=a(-b)^{-3/4}>0$, we obtain
$$
g(\xi,\theta)=\frac{1}{\xi^2}\left( \frac{3}{-\xi^{8/3}+3C\xi^2-3}d\xi^2+d\theta^2\right),
$$
for every $\xi\in J$, where $J$ is an open interval such that $-\xi^{8/3}+3C\xi^2-3>0$, for any positive $\xi\in J$ and $C$ a positive constant.

Next, we will determine the interval $J$. If we denote
$$
T(\xi)=-\xi^{8/3}+3C\xi^2-3, \qquad \xi>0,
$$
by straightforward computation, we get that $T(\xi)>0$ for any $\xi\in \left(\xi_{01},\xi_{02}\right)$, where $T\left(\xi_{01}\right)=T\left(\xi_{02}\right)=0$,
\begin{equation*}
\xi_{01}\in \left(0,\left(\frac{9}{4}C\right)^{3/2}\right) \quad \text{and}\quad
\xi_{02}\in \left(\left(\frac{9}{4}C\right)^{3/2}, \infty\right)
\end{equation*}
are the only positive vanishing points of $T$ and $C\in \left(\frac{4}{3^{3/2}},\infty\right)$.

Therefore, $\left(M^2,g\right)$ is isometric to $\left(D_C, g_C=\frac{3}{\xi^2\left(-\xi^{8/3}+3C\xi^2-3\right)}d\xi^2+\frac{1}{\xi^2}d\theta^2\right)$, $\ $ where $D_C=(\xi_{01},\xi_{02})\times\mathbb{R}$, $C\in \left(\frac{4}{3^{3/2}},\infty\right)$, and $\xi_{01}$ and $\xi_{02}$ are the vanishing points of $-\xi^{8/3}+3C\xi^2-3$, with $0<\xi_{01}<\xi_{02}$.
\end{proof}

The Riemannian surface $\left(D_C,g_C\right)$ has the following properties.

\begin{theorem}
\label{theorem_C=1}
Consider $\left(D_C, g_C\right)$. Then, we have
\begin{itemize}
  \item[(a)]$1-K_C(\xi,\theta)=\frac{1}{9}\xi^{8/3}>0,$  $K'_C(\xi)=-\frac{8}{27}\xi^{5/3}$ and $\grad K\neq 0$ at any point of $D_C$;
  \item[(b)] the immersion $\Phi_C:\left(D_C, g_C\right)\to \mathbb{S}^3$ given by
  $$
  \Phi_C(\xi,\theta)=\left(\sqrt{1-\frac{1}{C\xi^2}}\cos \zeta,\sqrt{1-\frac{1}{C\xi^2}}\sin \zeta,\frac{\cos(\sqrt{C}\theta)}{\sqrt{C}\xi},\frac{\sin(\sqrt{C}\theta)}{\sqrt{C}\xi}\right),
  $$
  is biconservative in $\mathbb{S}^3$, where
  $$
  \zeta(\xi)=\pm\int_{\xi_{00}}^\xi{\frac{\sqrt{C}\tau^{4/3}}{(-1+C\tau^2)\sqrt{-\tau^{8/3}+3C\tau^2-3}}d\tau}+c,
  $$
  and $c$ is a real constant.
\end{itemize}
\end{theorem}

\begin{proof}
Consider the Riemannian metric
$$
g_C=\frac{3}{\xi^2(-\xi^{8/3}+3C\xi^2-3)}d\xi^2+\frac{1}{\xi^2}d\theta^2
$$
on $D_C$ with coefficients given by
\begin{equation}
\label{eq:coef_I_form}
E_C=g_{C,11}=\frac{3}{\xi^2(-\xi^{8/3}+3C\xi^2-3)}, \quad F_C=g_{C,12}=0, \quad G_C=g_{C,22}=\frac{1}{\xi^2}.
\end{equation}
Using the formula of the Gaussian curvature
$$
K(\xi,\theta)=-\frac{1}{2\sqrt{EG}}\left(\frac{\partial}{\partial_\xi}\left(\frac{G_\xi}{\sqrt{EG}}\right)+\frac{\partial}{\partial_\theta}\left(\frac{E_\theta}{\sqrt{EG}}\right)\right),
$$
we obtain that $K_C$ is given by
$$
K_C(\xi,\theta)=K_C(\xi)=-\frac{1}{9}\xi^{8/3}+1
$$
and
$$
K'_C(\xi)=-\frac{8}{27}\xi^{5/3}.
$$
Therefore, $K'_C(\xi)<0$ at any $\xi\in \left(\xi_{01},\xi_{02}\right)$. Since
$$
(\grad K_C)(\xi,\theta)=\frac{\xi^2(-\xi^{8/3}+3C\xi^2-3)}{3}K'_C(\xi)\partial_\xi,
$$
we have that $|(\grad K)(\xi,\theta)|\neq 0$ for any $(\xi,\theta)\in D_C$.

To prove $(b)$, let us first recall that, if $M^2$ is a biconservative surface in $\mathbb{S}^3$, with $f>0$ and $\grad f\neq0$ at any point of $M$, then $M$ can be locally parameterized by
\begin{equation*}
\Phi_{C_1}(k,v)=\left(\sqrt{1-\frac{16}{9C_1}k^{-3/2}}\cos \mu, \sqrt{1-\frac{16}{9C_1}k^{-3/2}}\sin \mu, \frac{4\cos v}{3\sqrt{C_1}k^{3/4}},\frac{4\sin v}{3\sqrt{C_1}k^{3/4}}\right),
\end{equation*}
for any $(k,v)\in \left(k_{01},k_{02}\right)\times\mathbb{R}$, where $k_{01}$ and $k_{02}$ are the vanishing points of $-\frac{16}{9}k^2-16k^4+C_1k^{7/2}$, $k_{01}\in \left(0,\left(\frac{3}{64}C_1\right)^{2}\right)$, $k_{02}\in \left(\left(\frac{3}{64}C_1\right)^{2}, \infty\right)$, $C_1>\frac{64}{3^{5/4}}$, and
$$
\mu(k)=\pm \int_{k_0}^k \frac{108\sqrt{\frac{\tau^2} {-16+9C_1\tau^{3/2}}}}{\sqrt{\tau^{1/2}\frac{\left(-16+9C_1\tau^{3/2}\right)\left(9C_1\tau^{3/2}-16\left(1+9\tau^2\right)\right)}{C_1}}}d\tau+c_0,
$$
where $c_0$ is a real constant.

In order to compute the metric on this surface, we need the coefficients of the first fundamental form
\begin{equation*}
\begin{array}{cc}
E_{C_1}(k,v)=\frac{81C_1k^{3/2}-144}{k^2\left(9C_1k^{3/2}-16\right)\left(9C_1k^{3/2}-16\left(1+9k^2\right)\right)},\\\\
F_{C_1}(k,v)=0, \qquad  G_{C_1}(k,v)=\frac{16}{9C_1k^{3/2}}.
\end{array}
\end{equation*}
Thus, the Riemannian metric is given by
$$
g_{C_1}(k,v)=\frac{81C_1k^{3/2}-144}{k^2\left(9C_1k^{3/2}-16\right)\left(9C_1k^{3/2}-16\left(1+9k^2\right)\right)}dk^2+\frac{16}{9C_1k^{3/2}} dv^2.
$$

We write $C_1$ as $C_1=16\cdot 3^{1/4}C$, where $C\in\mathbb{R}^\ast_+$, and  we know that $C_1>\frac{64}{3^{5/4}}$, which implies $C>\frac{4}{3^{3/2}}$. Therefore, we can choose $C$ to be the positive constant for the metric $\left(D_C,g_C\right)$.

We note that we can consider the change of coordinates
$$
(k,v)=\left(3^{-3/2}\xi^{4/3},\frac{\sqrt{C_1}}{4\cdot 3^{1/8}}\theta\right),
$$
where $\xi$ and $\theta$ are the coordinates on the domain $D_C$. We have indeed
$$
-\xi^{8/3}+3C\xi^2-3=\frac{27}{16k^2}\left(-\frac{16}{9}k^2-16k^4+C_1k^{7/2}\right)
$$
and, therefore, the vanishing points $\xi_{01}$ and $\xi_{02}$ of $-\xi^{8/3}+3C\xi^2-3$ are the corresponding points to $k_{01}$ and $k_{02}$, i.e., $\xi_{01}=3^{9/8}k_{01}^{3/4}$ and $\xi_{02}=3^{9/8}k_{02}^{3/4}$.

Thus, we get the expression of the initial metric
$$
g_C(\xi,\theta)=\frac{3}{\xi^2\left(-\xi^{8/3}+3C\xi^2-3\right)}d\xi^2+\frac{1}{\xi^2}d\theta^2, \qquad (\xi,\theta)\in D_C.
$$
The local parametrization of the surface can be rewritten as
$$
\Phi_{C}(\xi,\theta)=\left(\sqrt{1-\frac{1}{C\xi^2}}\cos \zeta, \sqrt{1-\frac{1}{C\xi^2}}\sin \zeta, \frac{\cos(\sqrt{C}\theta)}{\sqrt{C}\xi},\frac{\sin(\sqrt{C}\theta)}{\sqrt{C}\xi} \right),
$$
for any $\xi\in \left(\xi_{01},\xi_{02}\right)$ and $\theta\in\mathbb{R}$, where $\zeta=\mu(k(\xi))$ is given by
$$
\zeta(\xi)=\pm\int_{\xi_{00}}^\xi{\frac{\sqrt{C}\tau^{4/3}}{\left(-1+C\tau^2\right)\sqrt{-\tau^{8/3}+3C\tau^2-3}}d\tau}+c,
$$
where $c\in \mathbb{R}$.
\end{proof}

\begin{remark}
The Gaussian curvature of $\left(D_C,g_C\right)$ does not depend on $C$.
\end{remark}

\begin{remark}
Since $\left(\grad K_C\right)(\xi,\theta)=-\frac{8\xi^{11/3}\left(-\xi^{8/3}+3C\xi^2-3\right)}{81}\partial_\xi$ for any $(\xi,\theta)\in D_C$, we get that
$$
\lim_{\xi\searrow\xi_{01}}\left(\grad K_C\right)(\xi,\theta)=\lim_{\xi\nearrow\xi_{02}}\left(\grad K_C\right)(\xi,\theta)=0.
$$
\end{remark}

Next, we denote 
\begin{equation*}
\zeta_0(\xi)=\int_{\xi_{00}}^\xi{\frac{\sqrt{C}\tau^{4/3}}{(-1+C\tau^2)\sqrt{-\tau^{8/3}+3C\tau^2-3}}d\tau}
\end{equation*}
and we state the the following lemma that we will use later in our paper. Its proof follows using standard arguments.

\begin{lemma}
\label{lemma_sign}
We have
$$
\quad \lim_{\xi\searrow \xi_{01}} \zeta_0(\xi)=\zeta_{0,-1}>-\infty \quad \textnormal{ and } \quad \lim_{\xi\nearrow \xi_{02}}\zeta_0(\xi)=\zeta_{0,1} < \infty.
$$
\end{lemma}

The following result shows that we do have a one-parameter family of Riemannian surfaces $\left(D_C,g_C\right)$.

\begin{proposition}
\label{isometry}
Let us consider $\left(D_C,g_C=\frac{3}{\xi^2\left(-\xi^{8/3}+3C\xi^2-3\right)}d\xi^2+\frac{1}{\xi^2}d\theta^2\right)$ and $\left(D_{\tilde{C}},g_{\tilde{C}}=\frac{3}{\tilde{\xi}^2\left(-\tilde{\xi}^{8/3}+3\tilde{C}\tilde{\xi}^2-3\right)}d\tilde{\xi}^2+\frac{1}{\tilde{\xi}^2}d\tilde{\theta}^2\right)$.
The surfaces $(D_C,g_C)$ and $\left(D_{\tilde{C}},g_{\tilde{C}}\right)$ are isometric if and only if $C=\tilde{C}$ and the isometry is $\Theta(\xi,\theta)=\left(\xi,\pm\theta+a_1\right)$, where $a_1$ is a real constant.
\end{proposition}

\begin{proof}
Assume that there exists an isometry $\Theta:(D_C,g_C)\to \left(D_{\tilde{C}},g_{\tilde{C}}\right)$ and denote $\Theta(\xi,\theta)=\left(\Theta^1(\xi,\theta),\Theta^2(\xi,\theta)\right)$. As we have seen in Theorem $\ref{theorem_C=1}$, the Gaussian curvature of $(D_C,g_C)$ is $K(\xi,\theta)=-\frac{1}{9}\xi^{8/3}+1$ and the Gaussian curvature of $\left(D_{\tilde{C}},g_{\tilde{C}}\right)$ is $\tilde{K}(\tilde{\xi},\tilde{\theta})=-\frac{1}{9}\tilde{\xi}^{8/3}+1$.

Since $\Theta$ is an isometry, we have that $\tilde{K}(\Theta(\xi,\theta))=K(\xi,\theta)$ and, taking into account the above expressions of the curvatures, we get $\Theta^1(\xi,\theta)=\xi>0$. Therefore, $\Theta(\xi,\theta)=\left(\xi,\Theta^2(\xi,\theta)\right)$.

Next, from $\left(\Theta^\ast g_{\tilde{C}}\right)\left(\partial_{\xi},\partial_{\xi}\right)=g_C\left(\partial_{\xi},\partial_{\xi}\right)$, i.e., $g_{\tilde{C}}\left(\Theta_\ast \partial_{\xi}, \Theta_\ast \partial_{\xi}\right)=g_C\left(\partial_{\xi},\partial_{\xi}\right)$, using $(\ref{eq:coef_I_form})$, we find
\begin{equation}
\label{eq:A_1}
\frac{3}{-\xi^{8/3}+3C\xi^2-3}=\frac{3}{-\xi^{8/3}+3\tilde{C}\xi^2-3}+\left(\frac{\partial \Theta^2}{\partial \xi}\right)^2.
\end{equation}
Similarly, from $\left(\Theta^\ast g_{\tilde{C}}\right)\left(\partial_{\xi},\partial_{\theta}\right)=g_C\left(\partial_{\xi},\partial_{\theta}\right)$ and $\left(\Theta^\ast g_{\tilde{C}}\right)\left(\partial_{\theta},\partial_{\theta}\right)=g_C\left(\partial_{\theta},\partial_{\theta}\right)$, using $(\ref{eq:coef_I_form})$, we get
\begin{equation}
\label{eq:A_2}
0=\frac{\partial \Theta^2}{\partial \xi}\cdot\frac{\partial \Theta^2}{\partial \theta} \quad \text{ and } \quad \frac{\partial \Theta^2}{\partial \theta}=\pm 1.
\end{equation}
From $(\ref{eq:A_2})$ one obtains $\frac{\partial \Theta^2}{\partial \xi}=0$. Now, using $(\ref{eq:A_1})$, it follows that $C=\tilde{C}$. Since $\frac{\partial \Theta^2}{\partial \xi}=0$ and $\frac{\partial \Theta^2}{\partial \theta}=\pm 1$, we have $\Theta(\xi,\theta)=(\xi,\pm\theta+a_1)$, where $a_1$ is a real constant.
\end{proof}

The Riemannian surface $\left(D_C,g_C\right)$ is not complete. In order to find a complete biconservative surface in $\mathbb{S}^3$, we will first construct a complete surface of revolution in $\mathbb{R}^3$. We begin with the following result.

\begin{theorem}
\label{theorem3.18}
Let us consider $\left(D_C,g_C\right)$, where $D_C=\left(\xi_{01},\xi_{02}\right)\times\mathbb{R}$ and $C\in \left(\frac{4}{3^{3/2}},\infty\right)$. Then $\left(D_C,g_C\right)$ is the universal cover of the surface of revolution in $\mathbb{R}^3$ given by
\begin{equation}
\label{psi(xi,theta)}
\Psi_{C,C^\ast}(\xi,\theta)=\left(f(\xi)\cos \frac{\theta}{C^\ast}, f(\xi)\sin \frac{\theta}{C^\ast},h(\xi)\right),
\end{equation}
where $f(\xi)=\frac{C^\ast}{\xi}$,
\begin{equation}
\label{h(xi)}
h(\xi)=\pm\int_{\xi_{00}}^\xi{\sqrt{\frac{3\tau^2-\left(C^\ast\right)^2\left(-\tau^{8/3}+3C\tau^2-3\right)}{\tau^4\left(-\tau^{8/3}+3C\tau^2-3\right)}}d\tau}+a,
\end{equation}
$C^\ast\in \left(0,\left(C-\frac{4}{3^{3/2}}\right)^{-1/2}\right)$ is a positive constant, $a$ is a real constant and $\xi_{00}$ is an arbitrary point in $\left(\xi_{01},\xi_{02}\right)$.
\end{theorem}

\begin{proof}
In fact, we can prove that if $\left(D_C,g_C\right)$ is (locally and intrinsically) isometric to a surface of revolution, then it has to be of the form $(\ref{psi(xi,theta)})$. To show this, let us consider
$$
\tilde{\Psi}\left(\tilde{\xi},\tilde{\theta}\right)=\left(\tilde{f}\left(\tilde{\xi}\right)\cos \tilde{\theta}, \tilde{f}\left(\tilde{\xi}\right)\sin \tilde{\theta},\tilde{h}\left(\tilde{\xi}\right)\right), \qquad \left(\tilde{\xi},\tilde{\theta}\right)\in \tilde{D},
$$
a surface of revolution, where $\tilde{D}$ is an open set in $\mathbb{R}^2$ and $\Theta:\left(D_C,g_C\right)\to \left(\tilde{D},\tilde{g}\right)$ an isometry, where
$$
\tilde{g}\left(\tilde{\xi},\tilde{\theta}\right)=\left(\left(\tilde{f}'\left(\tilde{\xi}\right)\right)^2+\left(\tilde{h}'\left(\tilde{\xi}\right)\right)^2\right) d\tilde{\xi}^2+\left(\tilde{f}\left(\tilde{\xi}\right)\right)^2d\tilde{\theta}^2.
$$
We will assume that $\tilde{f}\left(\tilde{\xi}\right)>0$ for any $\tilde{\xi}$.

Next, we will proceed in the same way as in the proof of Proposition $\ref{isometry}$. From $\tilde{K}(\Theta(\xi,\theta))=K(\xi,\theta)$, we get $\Theta^1(\xi,\theta)=\Theta^1(\xi)$. In order to simplify the notations, we write $\Theta^1=\tilde{\xi}$ and $\Theta^2=\tilde{\theta}$, so that $\tilde{\xi}(\xi,\theta)=\tilde{\xi}(\xi)$. As $\Theta^\ast\tilde{g}=g_C$, we get
\begin{equation}
\label{ec1}
\left(\frac{\partial \tilde{\theta}}{\partial\theta}\right)^2\left(f\left(\tilde{\xi}(\xi)\right)\right)^2=\frac{1}{\xi^2}
\end{equation}
and
\begin{equation}
\label{ec2}
\frac{\partial \tilde{\theta}}{\partial\theta} \frac{\partial \tilde{\theta}}{\partial\xi}\left(f\left(\tilde{\xi}(\xi)\right)\right)^2=0.
\end{equation}
From $(\ref{ec1})$, we get that $\frac{\partial \tilde{\theta}}{\partial\theta}\neq0$, and then, from $(\ref{ec2})$, it follows that $\frac{\partial \tilde{\theta}}{\partial\xi}=0$. Thus we have $\tilde{\theta}(\xi,\theta)=\tilde{\theta}(\theta)$.
Again from $(\ref{ec1})$, one obtains $\left(\frac{\partial \tilde{\theta}}{\partial\theta}\right)^2=\frac{1}{\xi^2\left(f\left(\tilde{\xi}(\xi)\right)\right)^2}$. Since the left hand term depends only on $\theta$ and the right hand term depends only on $\xi$, it follows that
\begin{equation}
\label{tildef}
\tilde{f}\left(\tilde{\xi}(\xi)\right)=\frac{C^\ast}{\xi},
\end{equation}
where $C^\ast\in \mathbb{R}^\ast_+$, and
\begin{equation*}
\tilde{\theta}(\theta)=\frac{\theta}{C^\ast}+a_0,
\end{equation*}
where $a_0\in \mathbb{R}$. In the following, we shall consider $a_0=0$.

Hence, we obtain
\begin{equation*}
\left(\left(\tilde{f}\circ \tilde{\xi}\right)'(\xi)\right)^2+ \left(\left(\tilde{h}\circ \tilde{\xi}\right)'(\xi)\right)^2=\frac{3}{\xi^2\left(-\xi^{8/3}+3C\xi^2-3\right)}
\end{equation*}
and, from $(\ref{tildef})$, one has
\begin{equation}
\label{ec3}
\left(\left(\tilde{h}\circ \tilde{\xi}\right)'(\xi)\right)^2=\frac{3\xi^2-\left(C^\ast\right)^2\left(-\xi^{8/3}+3C\xi^2-3\right)}{\xi^2\left(-\xi^{8/3}+3C\xi^2-3\right)}.
\end{equation}

Next, we have to find the conditions to be satisfied by the positive constant $C^\ast$, such that $3\xi^2-\left(C^\ast\right)^2\left(-\xi^{8/3}+3C\xi^2-3\right)>0$ for any $\xi\in \left(\xi_{01},\xi_{02}\right)$, where $C>\frac{4}{3^{3/2}}$ is fixed.
By standard arguments, it can be shown that if $C^\ast\in \left(0,\left(C-\frac{4}{3^{3/2}}\right)^{-1/2}\right)$, then the above inequality holds and
$$
\left(\tilde{h}\circ\tilde{\xi}\right)(\xi)=\pm\int_{\xi_{00}}^\xi {\sqrt{\frac{3\tau^2-\left(C^\ast\right)^2\left(-\tau^{8/3}+3C\tau^2-3\right)}{\tau^4\left(-\tau^{8/3}+3C\tau^2-3\right)}}d\tau}+a,
$$
for any $\xi\in \left(\xi_{01},\xi_{02}\right)$, where $a$ is a real constant.

Next, we consider $\Psi_{C,C^\ast}=\tilde{\Psi}\circ\Theta$ defined by
\begin{eqnarray*}
\Psi_{C,C^\ast}(\xi,\theta)&=&\left(\left(\tilde{f}\circ\tilde{\xi}\right)(\xi)\cos\left(\tilde{\theta}(\theta)\right),\left(\tilde{f}\circ\tilde{\xi}\right)(\xi)\sin\left(\tilde{\theta}(\theta)\right), \left(\tilde{h}\circ \tilde{\xi}\right)(\xi)\right) \\\\
&=&\left(f(\xi)\cos\frac{\theta}{C^\ast},f(\xi)\sin\frac{\theta}{C^\ast}, h'(\xi)\right), \qquad (\xi,\theta)\in D_C,
\end{eqnarray*}
where $C>\frac{4}{3^{3/2}}$ is a positive constant, $C^\ast\in \left(0,\sqrt{\frac{3^{3/2}}{3^{3/2}C-4}}\right)$, $f(\xi)=\frac{C^\ast}{\xi}$ and
$$
h(\xi)=\pm\int_{\xi_{00}}^\xi {\sqrt{\frac{3\tau^2-\left(C^\ast\right)^2\left(-\tau^{8/3}+3C\tau^2-3\right)}{\tau^4\left(-\tau^{8/3}+3C\tau^2-3\right)}}d\tau}+a,
$$
for any $\xi\in \left(\xi_{01},\xi_{02}\right)$, with $a$ a real constant.
\end{proof}

\begin{remark}
The mean curvature function of $\Psi_{C,C^\ast}$ is given by
$$
f_{C,C^\ast}=\frac{9\xi^2-\left(C^\ast\right)^2\left(-2\xi^{8/3}+9C\xi^2-18\right)}{6C^\ast\sqrt{9\xi^2-3\left(C^\ast\right)^2\left(-\xi^{8/3}+3C\xi^2-3\right)}}
$$
and we can see that it depends on both $C$ and $C^\ast$.
\end{remark}

\begin{remark}
From now on, we will take $\xi_{00}=\left(\frac{9}{4}C\right)^{3/2}\in \left(\xi_{01},\xi_{02}\right)$ and $C^\ast\in \left(0,\left(C-\frac{4}{3^{3/2}}\right)^{-1/2}\right)$.
\end{remark}

The function $h$ has the following properties which follows easily.

\begin{lemma}
Let
$$
h_0(\xi)=\int_{\xi_{00}}^\xi {\sqrt{\frac{3\tau^2-\left(C^\ast\right)^2\left(-\tau^{8/3}+3C\tau^2-3\right)}{\tau^4\left(-\tau^{8/3}+3C\tau^2-3\right)}}d\tau}, \qquad \xi\in \left(\xi_{01},\xi_{02}\right),
$$
i.e., we fix the sign in $(\ref{h(xi)})$ and we choose $a=a_0=0$. Then

\begin{itemize}
  \item [(a)] $\lim_{\xi\searrow \xi_{01}} h_0(\xi)=h_{0,-1}>-\infty$ \textnormal{and} $\lim_{\xi\nearrow \xi_{02}}h_0(\xi)=h_{0,1} < \infty$;

  \item [(b)] $h_0$ is strictly increasing and
  $$
  \lim_{\xi\searrow\xi_{01}}h_0'(\xi)=\lim_{\xi\nearrow\xi_{02}}h_0'(\xi)=\infty;
  $$
  \item [(c)] $\lim_{\xi\searrow \xi_{01}} h_0''(\xi)=-\infty$ \textnormal{and} $\lim_{\xi\nearrow \xi_{02}}h_0''(\xi) = \infty$.
\end{itemize}

\end{lemma}

We have shown that $\left(D_C,g_C\right)$ is isometric to the surface of revolution given by
$\Psi_{C,C^\ast}$ and this surface is not complete. Alternating the sign in $(\ref{h(xi)})$ and with appropriate choices of the constant $a$, we will construct a complete surface, which on an open dense subset is locally isometric to $\left(D_C,g_C\right)$. This surface is a surface of revolution whose profile curve is the graph of a function defined on the whole $Oh$ axis.

First, let us consider the profile curve $\sigma_0(\xi)=\left(f(\xi),h_0(\xi)\right)$, for any $\xi\in\left(\xi_{01},\xi_{02}\right)$. Obviously, $h_0:\left(\xi_{01},\xi_{02}\right)\to \left(h_{0,-1},h_{0,1}\right)$ is a diffeomorphism and we can consider $h_0^{-1}:\left(h_{0,-1},h_{0,1}\right) \to \left(\xi_{01},\xi_{02}\right)$, with $h_0^{-1}$ $:$ $\xi_0=\xi_0(h)$, $h\in \left(h_{0,-1},h_{0,1}\right)$. One can reparametrize $\sigma_0$ such that it becomes the graph of a function depending on the variable $h$,  $h\in\left(h_{0,-1},h_{0,1}\right)$.

In order to extend our surface in the upper part, we ask the line $h=h_{0,1}$ to be a symmetry axis. Therefore $2h_{0,1}=h_0(\xi)+h_1(\xi)$, where $h_1:\left(\xi_{01},\xi_{02}\right)\to\mathbb{R}$, and then we get $h_1(\xi)=2h_{0,1}-h_0(\xi)$; thus $a_1=2h_{0,1}$. It is easy to see that
$$
\lim_{\xi\searrow \xi_{01}} h_1(\xi)=2h_{0,1}-h_{0,-1}, \qquad \lim_{\xi\nearrow \xi_{02}}h_1(\xi) =h_{0,1},
$$
and, since $h_1'(\xi)=-h_0'(\xi)<0$, for any $\xi\in(\xi_{01},\xi_{02})$, it follows that $h_1$ is strictly decreasing and $h_1\left(\xi_{01},\xi_{02}\right)=\left(h_{0,1},2h_{0,1}-h_{0,-1}\right)$.
Since $h_1$ is a diffeomorphism on its image, we can consider $h^{-1}_1:\left(h_{0,1},2h_{0,1}-h_{0,-1}\right) \to \left(\xi_{01},\xi_{02}\right)$, with $h^{-1}_1$ $:$ $\xi_1=\xi_1(h)$, $ h\in \left(h_{0,1},2h_{0,1}-h_{0,-1}\right)$.

It is easy to see that
$$
\lim_{h\nearrow h_{0,1}}\xi_1(h)=\xi_{02}, \quad \lim_{h\searrow {2h_{0,1}-h_{0,-1}}} \xi_1(h)=\xi_{01},
$$
and, since $\left(h_1^{-1}\right)'(h)=\frac{1}{h_1'(\xi_1(h))}<0$, for any $h\in(h_{0,1},2h_{0,1}-h_{0,-1})$, it follows that $h_1^{-1}$ is strictly decreasing.

Next, we define a function $F_1:\left(h_{0,-1},2h_{0,1}-h_{0,-1}\right)\to\mathbb{R}$ by
\begin{equation*}
F_1(h)=
\left\{
\begin{array}{lll}
\xi_1(h)&,& h\in \left(h_{0,1},2h_{0,1}-h_{0,-1}\right)\\
\xi_{02}&,& h=h_{0,1}\\
\xi_0(h)&,& h\in \left(h_{0,-1},h_{0,1}\right)
\end{array}
\right.,
\end{equation*}
and we will prove that $F_1$ is at least of class $C^3$.

Obviously, $F_1$ is continuous.

In order to prove that $F_1$ is of class $C^1$, we first consider $h\in\left(h_{0,-1},h_{0,1}\right)$. In this case, we have
$$
F_1'(h)=\xi_0'(h)=\frac{1}{h_0'(\xi_0(h))}
$$
and
\begin{equation*}
\lim_{h\nearrow h_{0,1}}F_1'(h)=\lim_{h\nearrow h_{0,1}}\xi_0'(h)=\lim_{h\nearrow h_{0,1}}\frac{1}{h_0'(\xi_0(h))}=\lim_{\xi\nearrow \xi_{02}}\frac{1}{h_0'(\xi)}=0.
\end{equation*}
Then, we consider $h\in\left(h_{0,1},2h_{0,1}-h_{0,-1}\right)$, and we get
$$
F_1'(h)=\xi_1'(h)=\frac{1}{h_1'(\xi_1(h))}
$$
and
\begin{equation*}
\lim_{h\searrow h_{0,1}}F_1'(h)=\lim_{h\searrow h_{0,1}}\xi_1'(h)=\lim_{h\searrow h_{0,1}}\frac{1}{h_1'(\xi_1(h))}= \lim_{\xi\nearrow \xi_{02}}\frac{1}{h_1'(\xi)}=\lim_{\xi\nearrow \xi_{02}}\frac{1}{-h_0'(\xi)}=0.
\end{equation*}
Therefore, $\lim_{h\nearrow h_{0,1}}F_1'(h)=\lim_{h\searrow h_{0,1}}F_1'(h)=0\in\mathbb{R}$, which means that there exists $F_1'(h_{0,1})=0$ and $F_1$ is of class $C^1$.

In a similar way, one can prove that $F_1$ is of class $C^2$ and $C^3$.

In order to extend our surface in the lower part, we ask the line $h=h_{0,-1}$ to be a symmetry axis. Therefore, $2h_{0,-1}=h_0(\xi)+h_{-1}(\xi)$, where $h_{-1}:\left(\xi_{01},\xi_{02}\right)\to\mathbb{R}$, and we get $h_{-1}(\xi)=2h_{0,-1}-h_0(\xi)$; thus $a_{-1}=2h_{0,-1}$. It is easy to see that
$$
\lim_{\xi\nearrow \xi_{02}}h_{-1}(\xi) =2h_{0,-1}-h_{0,1}, \qquad \lim_{\xi\searrow \xi_{01}} h_{-1}(\xi)=h_{0,-1},
$$
and, since $h_{-1}'(\xi)=-h_0'(\xi)<0$, for any $\xi\in(\xi_{01},\xi_{02})$, it follows that $h_{-1}$ is strictly decreasing and $h_{-1}\left(\xi_{01},\xi_{02}\right)=\left(2h_{0,-1}-h_{0,1}, h_{0,-1}\right)$.
Since $h_{-1}$ is a diffeomorphism on its image, we can consider $h^{-1}_{-1}:\left(2h_{0,-1}-h_{0,1}, h_{0,-1}\right) \to \left(\xi_{01},\xi_{02}\right)$, with $h^{-1}_{-1}$ $:$ $\xi_{-1}=\xi_{-1}(h)$, $ h\in \left(2h_{0,-1}-h_{0,1}, h_{0,-1}\right)$.

It is easy to see that
$$
\lim_{h\nearrow 2h_{0,-1}-h_{0,1}}\xi_{-1}(h)=\xi_{02}, \qquad \lim_{h\searrow {h_{0,-1}}} \xi_{-1}(h)=\xi_{01},
$$
and, since $\left(h_{-1}^{-1}\right)'(h)=\frac{1}{h_{-1}'(\xi_{-1}(h))}<0$, for any $h\in(2h_{0,-1}-h_{0,1},h_{0,-1})$, we get that $h_{-1}^{-1}$ is strictly decreasing.

Further, we define the function $F_{-1}:\left(2h_{0,-1}-h_{0,1}, h_{0,1}\right)\to\mathbb{R}$ by
\begin{equation*}
F_{-1}(h)=
\left\{
\begin{array}{lll}
\xi_0(h)&,& h\in \left(h_{0,-1},h_{0,1}\right) \\
\xi_{01}&,& h=h_{0,-1}\\
\xi_{-1}(h)&,& h\in \left(2h_{0,-1}-h_{0,1},h_{0,-1}\right)
\end{array}
\right.,
\end{equation*}
and, in a similar way to the proof of $C^3$ smoothness of $F_1$, we can show that also $F_{-1}$ is at least of class $C^3$.

Now, we extend the functions $F_1$ and $F_{-1}$ to the whole line $\mathbb{R}$. This construction will be done by symmetry to the lines $h=h_{0,k}$, $k\in \mathbb{Z}^\ast$.

We define $h_{0,2}=2h_{0,1}-h_{0,-1}$, $h_{0,3}=2h_{0,2}-h_{0,1}=3h_{0,1}-2h_{0,-1}$, etc; then $h_{0,-2}=2h_{0,-1}-h_{0,1}$, $h_{0,-3}=2h_{0,-2}-h_{0,-1}=3h_{0,-1}-2h_{0,1}$, etc. This way we obtain
\begin{equation*}
h_{0,k}=
\left\{
\begin{array}{lll}
k\  h_{0,1}-(k-1)h_{0,-1} &,& k\geq 1 \\
-k\  h_{0,-1}+(k+1)h_{0,1} &,& k\leq - 1
\end{array}
\right. .
\end{equation*}
The functions $h_k$ are obtained in the same way. For example, $h_1(\xi)=2h_{0,1}-h_0(\xi)$, $h_2(\xi)=2h_{0,2}-h_1(\xi)=2h_{0,1}-2h_{0,-1}+h_0(\xi)$, etc; then $h_{-1}(\xi)=2h_{0,-1}-h_0(\xi)$, $h_{-2}(\xi)=2h_{0,-2}-h_{-1}(\xi)=2h_{0,-1}-2h_{0,1}+h_0(\xi)$, etc. In general, we have
\begin{equation*}
h_{k}(\xi)=
\left\{
\begin{array}{lll}
2 h_{0,k}-h_{k-1}(\xi)&,& k\geq 1 \\
2 h_{0,k}-h_{k+1}(\xi) &,& k\leq - 1
\end{array}
\right. .
\end{equation*}
We note that for $h_k$ we have the following formulas
\begin{equation*}
h_{k}(\xi)=
\left\{
\begin{array}{lll}
k\left( h_{0,1}-h_{0,-1}\right)+h_0(\xi)&,& k=2p, \   p\in \mathbb{Z} \\
(k+1)h_{0,1}-(k-1)h_{0,-1}-h_0(\xi) &,& k=2p+1, \  p\in\mathbb{Z}
\end{array}
\right. .
\end{equation*}
Denoting the inverse of the function $h_k$ by $\xi_k$, we define the function
\begin{equation*}
F(h)=
\left\{
\begin{array}{lll}
\xi_{01} &,& h=h_{0,k}, \ k=2p, \ p\geq 1 \\
\xi_{02} &,& h=h_{0,k}, \ k=2p+1, \ p\geq 0\\
\xi_k(h) &,& h\in \left(h_{0,k},h_{0,k+1}\right), \ k\geq 1 \\
\xi_{02} &,& h=h_{0,1}\\
\xi_0(h) &,& h\in \left(h_{0,-1},h_{0,1}\right) \\
\xi_{01} &,& h=h_{0,-1}\\
\xi_{k}(h)&,& h\in \left(h_{0,k-1},h_{0,k}\right), k\leq -1\\
\xi_{01} &,& h=h_{0,k}, \ k=2p-1, \ p\leq 0 \\
\xi_{02} &,& h=h_{0,k}, \ k=2p, \ p\leq -1
\end{array}
\right. ,
\end{equation*}
which is at least of class $C^3$.

\begin{remark}
When $C=C^\ast=1$, $a=0$ and $\xi_{00}=\left(\frac{9}{4}\right)^{3/2}$, the plots of
$$
h_0(\xi)=\int_{\xi_{00}}^\xi {\sqrt{\frac{3\tau^2-\left(-\tau^{8/3}+3C\tau^2-3\right)}{\tau^4\left(-\tau^{8/3}+3C\tau^2-3\right)}}d\tau},
$$
$h_1(\xi)=2h_{02}-h(\xi)$, $h_{-1}(\xi)=2h_{01}-h(\xi)$, and of corresponding profile curves $\sigma_0(\xi)=\left(\frac{1}{\xi},h_0(\xi)\right)$, $\sigma_1(\xi)=\left(\frac{1}{\xi},h_1(\xi)\right)$, and $\sigma_{-1}(\xi)=\left(\frac{1}{\xi},h_{-1}(\xi)\right)$, for $\xi\in\left(\xi_{01},\xi_{02}\right)$, are as follows

\begin{tabular}{c c}
\begin{tikzpicture}[xscale=1,yscale=1]

\node[inner sep=0pt]  at (3,-1.4){\includegraphics[width=.4\textwidth]{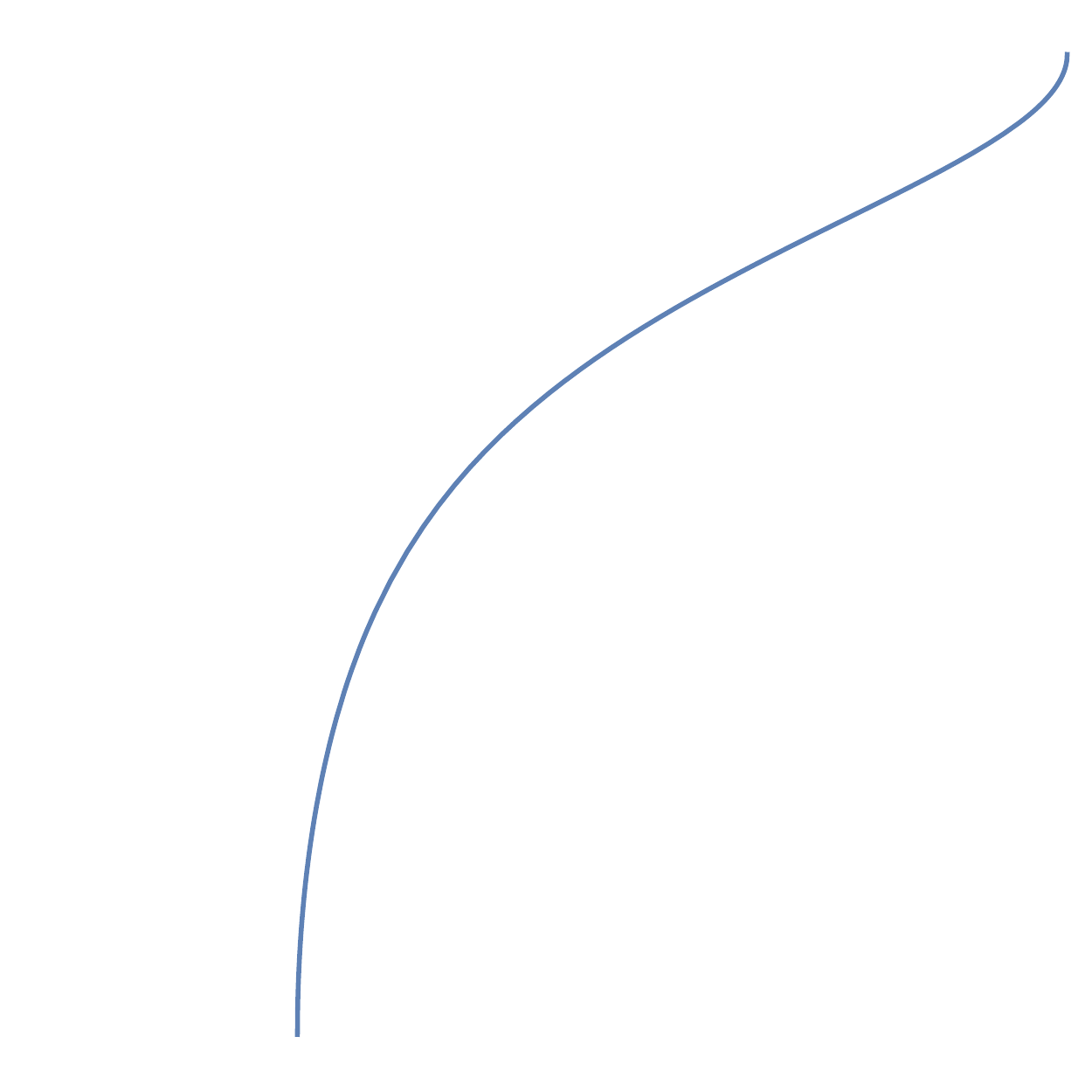}};
\draw [->,>=triangle 45] (0,0) -- (6,0);
\draw [->,>=triangle 45] (0,-4.1) -- (0,3);
\draw (6,0) node[below] {$\xi$};
\draw (0,3) node[left] {$h$};
\end{tikzpicture}
&
\begin{tikzpicture}[xscale=1,yscale=1]

\node[inner sep=0pt]  at (3,-.5){\includegraphics[width=.4\textwidth]{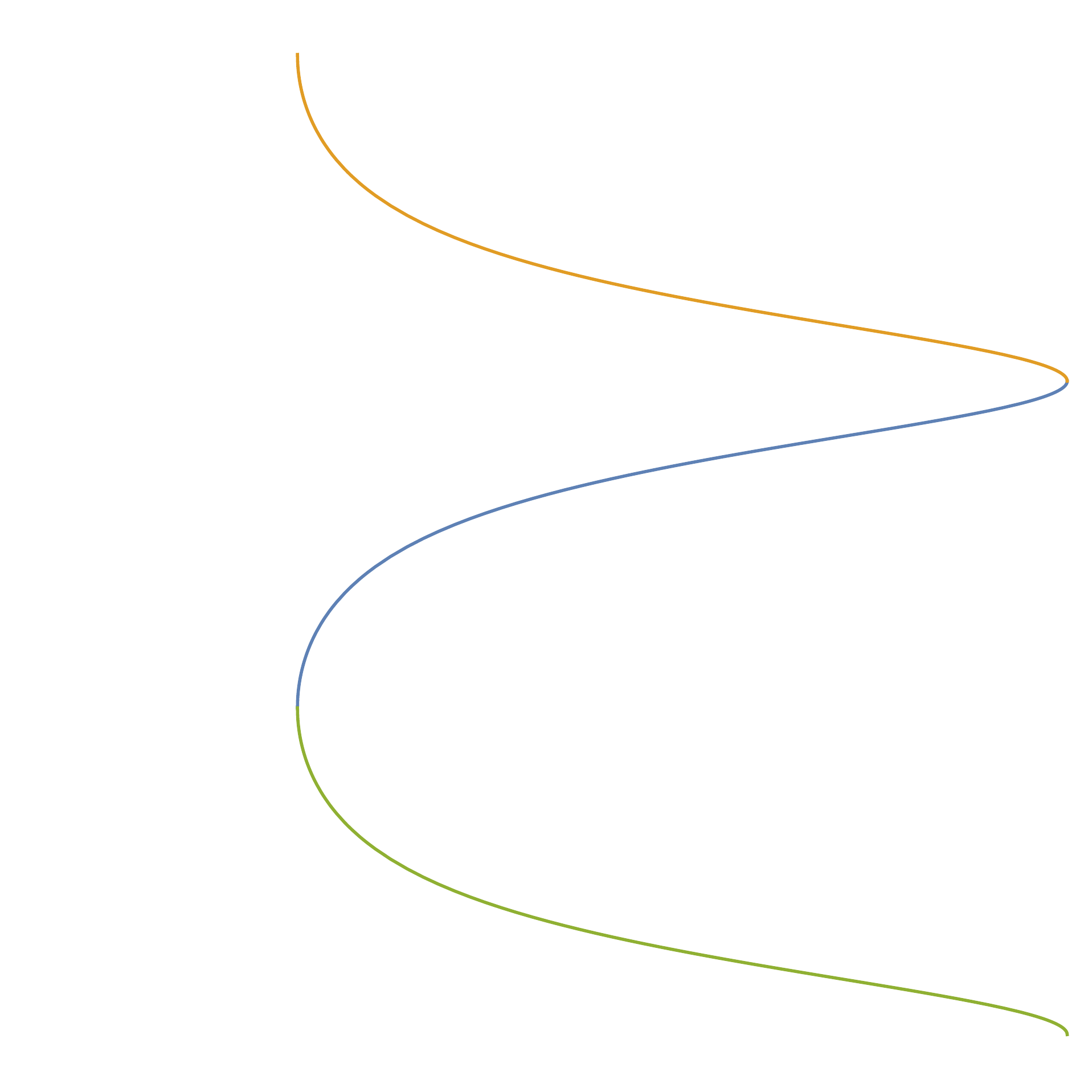}};
\draw [->,>=triangle 45] (0,0) -- (6,0);
\draw [->,>=triangle 45] (0,-4.1) -- (0,3);
\draw (6,0) node[below] {$\xi$};
\draw (0,3) node[left] {$h$};
\end{tikzpicture}
\\
Figure 2. Plot of $h_0$. &
Figure 3. Plot of $h_0$, $h_1$ and $h_{-1}$.
\\
\end{tabular}

\begin{tabular}{c c}
\begin{tikzpicture}[xscale=1,yscale=1]

\node[inner sep=0pt]  at (3,-1.4){\includegraphics[width=.4\textwidth]{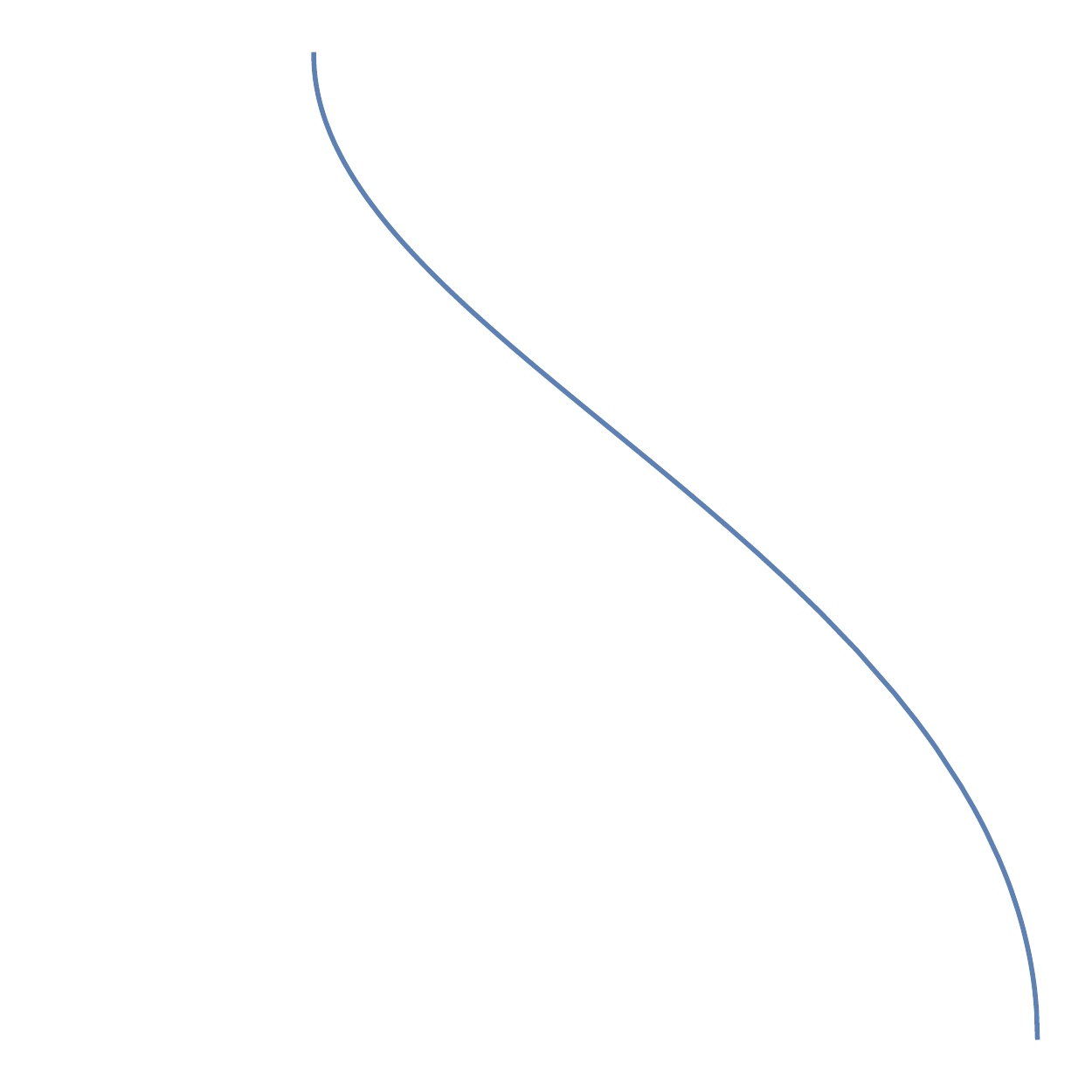}};
\draw [->,>=triangle 45] (0,0) -- (6,0);
\draw [->,>=triangle 45] (0,-4.1) -- (0,3);
\draw (6,0) node[below] {$f$};
\draw (0,3) node[left] {$h$};
\end{tikzpicture}
&
\begin{tikzpicture}[xscale=1,yscale=1]

\node[inner sep=0pt]  at (3,-.5){\includegraphics[width=.4\textwidth]{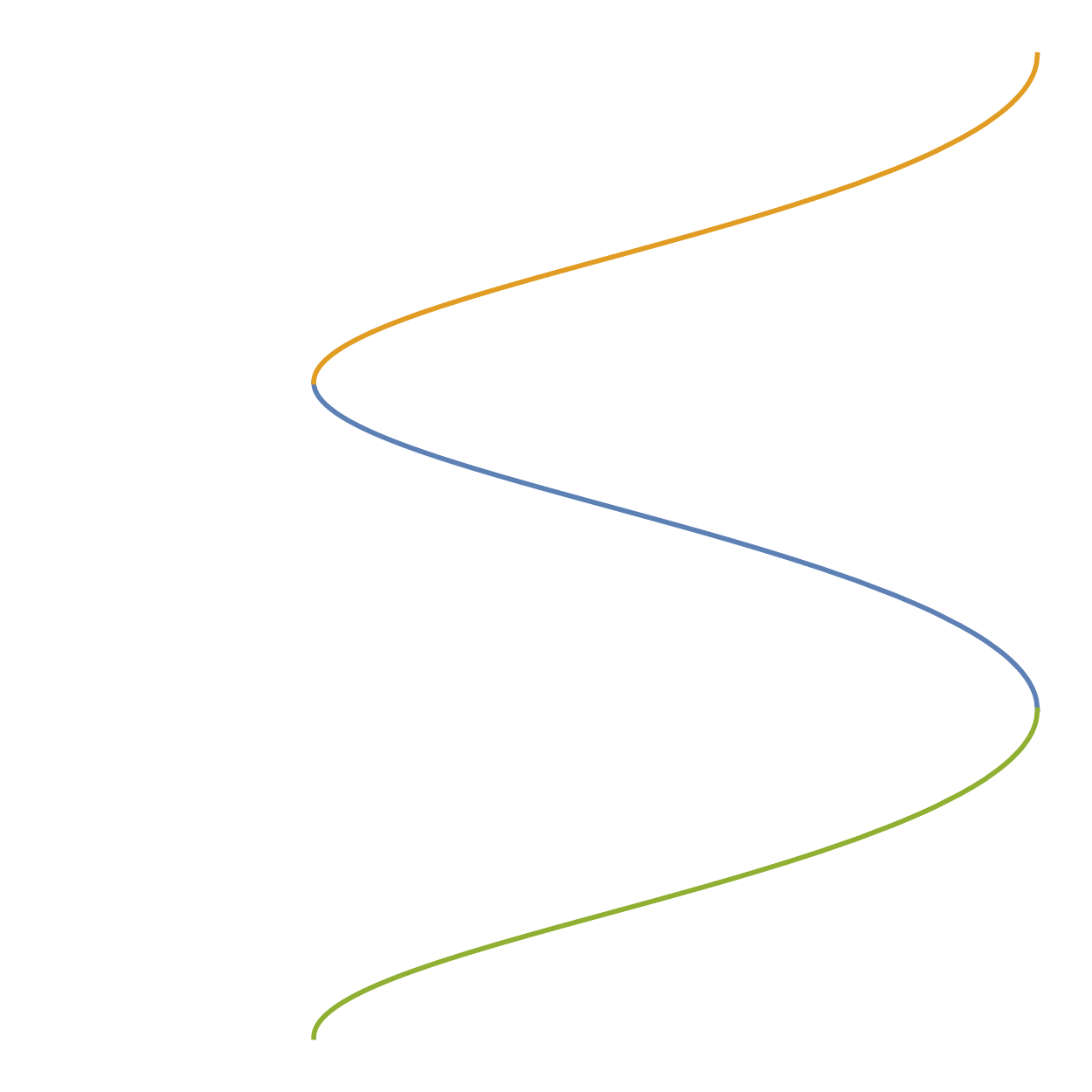}};
\draw [->,>=triangle 45] (0,0) -- (6,0);
\draw [->,>=triangle 45] (0,-4.1) -- (0,3);
\draw (6,0) node[below] {$f$};
\draw (0,3) node[left] {$h$};
\end{tikzpicture}
\\
Figure 4. Plot of $\sigma_0$.&
Figure 5. Plot of $\sigma_0$, $\sigma_1$ and $\sigma_{-1}$.
\\
\end{tabular}

\end{remark}

\begin{remark}
The function $F$ is periodic with main period $2\left(h_{0,1}-h_{0,-1}\right)$.
\end{remark}

\begin{remark}
The function $F$ depends on $C$ and $C^\ast$.
\end{remark}

We define $\sigma_k(\xi)=\left(f(\xi),h_k(\xi)\right)$, $\xi\in \left(\xi_{01},\xi_{02}\right)$, where $k\in \mathbb{Z}$. From Theorem $\ref{theorem3.18}$, we know that $\left(D_C, g_C\right)$ is isometric to the surface of revolution given by
$$
\Psi_{C,C^\ast}(\xi,\theta)=\left(f(\xi)\cos \frac{\theta}{C^\ast}, f(\xi) \sin \frac{\theta}{C^\ast}, h_k(\xi)\right), \qquad (\xi,\theta)\in D_C.
$$
We can reparametrize $\sigma_k$ and one obtains
{\small{
\begin{equation*}
\sigma_k(h)=\left\{
\begin{array}{llcl}
\sigma\left(\xi_k(h)\right)=\left((f\circ \xi_k)(h), h\right)=\left((f\circ F)(h),h\right), &h\in\left(h_{0,k},h_{0,k+1}\right)&,&  k\geq 1 \\
\sigma\left(\xi_0(h)\right)=\left((f\circ \xi_0)(h), h\right)=\left((f\circ F)(h),h\right), &h\in\left(h_{0,-1},h_{0,1}\right)&,& k=0 \\
\sigma\left(\xi_{k}(h)\right)=\left((f\circ \xi_{k})(h), h\right)=\left((f\circ F)(h),h\right),& h\in\left(h_{0,k-1},h_{0,k}\right)&,& k\leq -1
\end{array}
\right..
\end{equation*}
}}
Now, let us consider the profile curve
$$
\sigma(h)=\left((f\circ F)(h),h\right), \qquad h\in \mathbb{R}.
$$
Of course, $\sigma$ is the graph of $f\circ F$, it is at least of class $C^3$ and periodic. We can state the following theorem.

\begin{theorem}
\label{main_th3}
The surface of revolution given by
\begin{equation*}
\Psi_{C,C^\ast}(h,\theta)=\left((f\circ F)(h)\cos\frac{\theta}{C^\ast}, (f\circ F)(h)\sin \frac{\theta}{C^\ast},h \right), \qquad (h,\theta)\in\mathbb{R}^2,
\end{equation*}
is complete and, on an open dense subset, it is locally isometric to $\left(D_C, g_C\right)$. The induced metric is given by
$$
g_{C, C^\ast}(h,\theta)=\frac{3F^2(h)}{3F^2(h)-\left(C^\ast\right)^2(-F^{8/3}(h)+3CF^2(h)-3)}dh^2+\frac{1}{F^2(h)}d\theta^2,
$$
$(h,\theta)\in\mathbb{R}^2$. Moreover, $\grad K\neq 0$ at any point of that open dense subset, and $1-K>0$ everywhere.
\end{theorem}

From Theorem $\ref{main_th3}$ we easily get the following result.

\begin{proposition}
The universal cover of the surface of revolution given by $\Psi_{C,C^\ast}$ is $\mathbb{R}^2$ endowed with the metric $g_{C,C^\ast}$. It is complete, $1-K>0$ on $\mathbb{R}^2$ and, on an open dense subset, it is locally isometric to $\left(D_C,g_C\right)$ and $\grad K\neq 0$ at any point. Moreover any two $\left(\mathbb{R}^2,g_{C,C_1^\ast}\right)$ and $\left(\mathbb{R}^2,g_{C,C_2^\ast}\right)$ are isometric.
\end{proposition}

\begin{proof}
We only have to prove the last statement. We construct the isometry between $\left(\mathbb{R}^2,g_{C,C_1^\ast}\right)$ and $\left(\mathbb{R}^2,g_{C,C_2^\ast}\right)$ in a natural way, in the sense that, for example, it maps the interval $\left(h_{0,-1},h_{0,1}\right)$ corresponding to $C_1^\ast$ onto the interval $\left(h_{0,-1},h_{0,1}\right)$ corresponding to $C_2^\ast$. Repeating this process, we obtain an (at least) $C^3$ diffeomorphism of $\mathbb{R}^2$. It is easy to see that such diffeomorphism is a global isometry.
\end{proof}

From Theorem \ref{theorem_C=1} and Lemma $\ref{lemma_sign}$, we have that
$\Phi_C:\left(D_C, g_C\right)\to \mathbb{S}^3$,
$$
\Phi_C(\xi,\theta)=\left(\sqrt{1-\frac{1}{C\xi^2}}\cos \zeta,\sqrt{1-\frac{1}{C\xi^2}}\sin \zeta,\frac{\cos(\sqrt{C}\theta)}{\sqrt{C}\xi},\frac{\sin(\sqrt{C}\theta)}{\sqrt{C}\xi}\right),
$$
with $\zeta(\xi)=\pm\left(\zeta_0(\xi)+c\right),$ is a biconservative immersion in $\mathbb{S}^3$ and
$$
\lim_{\xi\searrow \xi_{01}} \zeta_0(\xi)=\zeta_{0,-1}>-\infty, \qquad \lim_{\xi\nearrow \xi_{02}}\zeta_0(\xi)=\zeta_{0,1} < \infty.
$$

In the last part of our paper we will construct a biconservative immersion from $\left(\mathbb{R}^2,g_{C,C^\ast}\right)$ in $\mathbb{S}^3$, as we claimed at the beginning of this section.

In order to do this, starting with the first component of the parametrization, we consider the following continuous functions defined on $\left[\xi_{01},\xi_{02}\right]$:
\begin{equation*}
\Phi^1_k(\xi)=
\left\{
\begin{array}{lll}
\sqrt{1-\frac{1}{C\xi^2}}\cos \left(\zeta_0(\xi)+c_k\right)&,& \xi\in \left(\xi_{01},\xi_{02}\right) \\
\sqrt{1-\frac{1}{C\xi_{01}^2}}\cos \left(\zeta_{0,-1}+c_k\right)&,& \xi=\xi_{01}\\
\sqrt{1-\frac{1}{C\xi_{02}^2}}\cos \left(\zeta_{0,1}+c_k\right)&,& \xi=\xi_{02}
\end{array}
\right.,
\end{equation*}
where $c_k\in\mathbb{R}$ for any $k\in \mathbb{Z}$.

Next, consider the function $\Phi^1:\mathbb{R}\to\mathbb{R}$ defined by
\begin{equation}
\label{p1}
\Phi^1(h)=
\left\{
\begin{array}{lllll}
\left(\Phi_k^1\circ F \right)(h)&,& h\in \left[h_{0,k},h_{0,k+1}\right]&,& k\geq 1 \\
\left(\Phi_0^1\circ F \right)(h)&,& h\in \left[h_{0, -1},h_{0,1}\right] \\
\left(\Phi_{k}^1\circ F \right)(h)&,& h\in \left[h_{0,k-1},h_{0,k}\right]&,& k\leq -1
\end{array}
\right. .
\end{equation}
We will prove that $\Phi^1$ is of class $C^3$. Since $F$ is a periodic function, with  main period $2\left(h_{0,1}-h_{0,-1}\right)$, it is enough to ask $\Phi^1$ to be a $C^3$ function on the interval $\left(h_{0,-2},h_{0,2}\right)=\left(2h_{0,-1}-h_{0,1},2h_{0,1}-h_{0,-1}\right)$. This means that it is enough to study the behaviour of $F$ at $h_{0,-1}$ and $h_{0,1}$.

First, we ask $\Phi^1$ to be continuous at $h_{0,-1}$ and $h_{0,1}$, i.e.,
\begin{equation*}
\lim_{h\nearrow h_{0,1}}\Phi^1(h)=\lim_{h\searrow h_{0,1}}\Phi^1(h)\in\mathbb{R}, \qquad \lim_{h\searrow h_{0,-1}}\Phi^1(h)=\lim_{h\nearrow h_{0,-1}}\Phi^1(h)\in\mathbb{R}.
\end{equation*}
Since
\begin{eqnarray*}
\lim_{h\nearrow h_{0,1}}\Phi^1(h)&=&\lim_{h\nearrow h_{0,1}}\Phi^1_0(F(h))=\lim_{h\nearrow h_{0,1}}\Phi^1_0(\xi_0(h))\\
&=&\lim_{\xi\nearrow \xi_{02}}\Phi^1_0(\xi)=\sqrt{1-\frac{1}{C\xi_{02}^2}}\cos \left(\zeta_{0,1}+c_0\right)\in\mathbb{R}
\end{eqnarray*}
and
\begin{eqnarray*}
\lim_{h\searrow h_{0,1}}\Phi^1(h)&=&\lim_{h\searrow h_{0,1}}\Phi^1_1(F(h))=\lim_{h\searrow h_{0,1}}\Phi^1_1(\xi_1(h))\\
&=&\lim_{\xi\nearrow \xi_{02}}\Phi^1_1(\xi)=\sqrt{1-\frac{1}{C\xi_{02}^2}}\cos \left(\zeta_{0,1}+c_1\right)\in\mathbb{R},
\end{eqnarray*}
we get that $\cos \left(\zeta_{0,1}+c_0\right)=\cos \left(\zeta_{0,1}+c_1\right)$. Therefore, we have two cases, as $c_1=c_0+2s_1\pi$ or $c_1=-2\zeta_{0,1}-c_0+2s_1\pi$, where $s_1\in \mathbb{Z}$, i.e.,
$$
c_1\equiv c_0\left(\mod 2\pi\right)\quad \text{ or }\quad c_1\equiv \left(-2\zeta_{0,1}-c_0\right)\left(\mod 2\pi\right).
$$

In a similar way, for $h_{0,-1}$, we have
\begin{eqnarray*}
\lim_{h\searrow h_{0,-1}}\Phi^1(h)&=&\lim_{h\searrow h_{0,-1}}\Phi^1_0(F(h))=\lim_{h\searrow h_{0,-1}}\Phi^1_0(\xi_0(h))\\
&=&\lim_{\xi\searrow \xi_{01}}\Phi^1_0(\xi)=\sqrt{1-\frac{1}{C\xi_{01}^2}}\cos \left(\zeta_{0,-1}+c_0\right)\in\mathbb{R}
\end{eqnarray*}
and
\begin{eqnarray*}
\lim_{h\nearrow h_{0,-1}}\Phi^1(h)&=&\lim_{h\nearrow h_{0,-1}}\Phi^1_1(F(h))=\lim_{h\nearrow h_{0,-1}}\Phi^1_1(\xi_{-1}(h))\\
&=&\lim_{\xi\nearrow \xi_{01}}\Phi^1_1(\xi)=\sqrt{1-\frac{1}{C\xi_{01}^2}}\cos \left(\zeta_{0,-1}+c_{-1}\right)\in\mathbb{R}.
\end{eqnarray*}
Hence, we must have $\cos \left(\zeta_{0,-1}+c_0\right)=\cos \left(\zeta_{0,-1}+c_{-1}\right)$. Therefore we again have two cases as $c_{-1}=c_0+2s_{-1}\pi$ or $c_{-1}=-2\zeta_{0,-1}-c_0+2s_{-1}\pi$, where $s_{-1}\in \mathbb{Z}$, i.e., $c_{-1}\equiv c_0\left(\mod 2\pi\right)$ or $c_{-1}\equiv \left(-2\zeta_{0,-1}-c_0\right)\left(\mod 2\pi\right)$.

By some straightforward computation, we can see that $\Phi^1$ is of class $C^1$ on the interval $\left(h_{0,-2},h_{0,2}\right)$ if and only if
$$
\sin \left(\zeta_{0,1}+c_0\right)= - \sin \left(\zeta_{0,1}+c_1\right) \quad\text{ and } \quad \sin \left(\zeta_{0,-1}+c_0\right)= - \sin \left(\zeta_{0,-1}+c_{-1}\right).
$$
We recall that, from the continuity of $\Phi^1$, there are two possibilities for each $c_1$ and $c_{-1}$ and we can then choose
\begin{equation*}
c_1\equiv \left(-2\zeta_{0,1}-c_0\right)\left(\mod 2\pi\right) \text { and } c_{-1}\equiv \left(-2\zeta_{0,-1}-c_0\right)\left(\mod 2\pi\right).
\end{equation*}
With this choice, one obtains that $\Phi^1$ is of class $C^3$ on $\left(h_{0,-2}, h_{0,2}\right)$.

In general, if we ask $\Phi^1$ to be of class $C^3$ on $\mathbb{R}$, since $F$ is periodic, it can be shown that we have the following relations between two consecutive $c_k$, where $k\in\mathbb{Z}$:
\begin{equation}
  \label{c_k1}
  c_k\equiv
  \left\{
    \begin{array}{lllll}
       \left(-2\zeta_{0,1}-c_{k-1}\right)&\left(\mod 2\pi\right),&  k=2p+1 &,& p\in \mathbb{N}\\
       \left(-2\zeta_{0,-1}-c_{k-1}\right)&\left(\mod 2\pi\right),& k=2p &,& p \in\mathbb{N}\\
       \left(-2\zeta_{0,-1}-c_{k+1}\right)&\left(\mod 2\pi\right),& k=2p-1 &,& p\in \mathbb{Z}_{-}\\
       \left(-2\zeta_{0,1}-c_{k+1}\right)&\left(\mod 2\pi\right),&  k=2p &,& p\in \mathbb{Z}_{-}
    \end{array}
  \right.,
\end{equation}
or, equivalently,
\begin{equation*}
  c_k\equiv
  \left\{
    \begin{array}{lllll}
       \left(-2\zeta_{0,1}-c_{k-1}\right)&\left(\mod 2\pi\right),& k=2p+1 &,& p\in \mathbb{Z}\\
       \left(-2\zeta_{0,-1}-c_{k-1}\right)&\left(\mod 2\pi\right),& k=2p &,& p \in\mathbb{Z}
    \end{array}
  \right..
\end{equation*}

We note that for $c_k$, we also have the following formulas
\begin{equation}
  \label{c_k2}
  c_k\equiv
  \left\{
    \begin{array}{lllll}
      \left( k\left(\zeta_{0,1}-\zeta_{0,-1}\right)+c_0\right) &\left(\mod 2\pi\right),& k=2p &,& p\in \mathbb{Z}\\
     \left((k-1)\zeta_{0,-1} -(k+1)\zeta_{0,1}-c_0\right)&\left(\mod 2\pi\right),& k=2p+1 &,& p \in\mathbb{Z}
    \end{array}
  \right..
\end{equation}

To study the second component of the parametrization $\Phi_C$, we will work in a similar way as for the first one. We consider the following continuous functions defined on $\left[\xi_{01},\xi_{02}\right]$:
\begin{equation*}
\Phi^2_k(\xi)=
\left\{
\begin{array}{lll}
(-1)^k\sqrt{1-\frac{1}{C\xi^2}}\sin \left(\zeta_{0}(\xi)+c_k\right)&,& \xi\in \left(\xi_{01},\xi_{02}\right) \\
(-1)^k\sqrt{1-\frac{1}{C\xi_{01}^2}}\sin \left(\zeta_{0,-1}+c_k\right)&,& \xi=\xi_{01}\\
(-1)^k\sqrt{1-\frac{1}{C\xi_{02}^2}}\sin \left(\zeta_{0,1}+c_k\right)&,& \xi=\xi_{02}
\end{array}
\right.,
\end{equation*}
where $c_k\in\mathbb{R}$, for any $k\in \mathbb{Z}$, are given by $(\ref{c_k1})$.

Then, we consider the function $\Phi^2:\mathbb{R}\to\mathbb{R}$ defined by
\begin{equation}
\label{p2}
\Phi^2(h)=
\left\{
\begin{array}{lll}
\left(\Phi_k^2\circ F \right)(h)&,& h\in \left[h_{0,k},h_{0,k+1}\right], k\geq 1 \\
\left(\Phi_0^2\circ F \right)(h)&,& h\in \left[h_{0, -1},h_{0,1}\right] \\
\left(\Phi_{k}^2\circ F \right)(h)&,& h\in \left[h_{0,k-1},h_{0,k}\right], k\leq -1
\end{array}
\right. .
\end{equation}

It can be shown that, with these choices of the constants $c_k$, $\Phi^2$ is of class $C^3$.  The proof is similar to the proof of $C^3$ smoothness of $\Phi^1$.

For the third component of the parametrization $\Phi_C$, we consider the following function
$$
\Phi_0^3(\xi)=\frac{1}{\sqrt{C}\xi}, \qquad \xi\in \left[\xi_{01},\xi_{02}\right],
$$
It is obvious that $\Phi_0^3$ is a smooth function on $\left[\xi_{01},\xi_{02}\right]$.

Let us consider a new function $\Phi^3:\mathbb{R}\to\mathbb{R}$ defined by
\begin{equation}
\label{p3}
\Phi^3(h)=(\Phi_0^3\circ F)(h), \qquad h\in \mathbb{R}.
\end{equation}
Since $F$ is at least of class $C^3$ on $\mathbb{R}$  and $\Phi_0^3$ is  smooth on $\left[\xi_{01},\xi_{02}\right]$, it follows that $\Phi^3$ is at least of class $C^3$ on $\mathbb{R}$.

For the forth component of the parametrization $\Phi_C$, we define $\Phi^4$ as $\Phi^3$, i.e.,
\begin{equation}
\label{p4}
\Phi^4(h)=(\Phi_0^4\circ F)(h), \qquad h\in \mathbb{R},
\end{equation}
where $\Phi_0^4(\xi)=\frac{1}{\sqrt{C}\xi}$, for any $\xi\in \left[\xi_{01},\xi_{02}\right]$.

Now, we can conclude with the following theorem.

\begin{theorem}
\label{main_th2}
The map $\Phi_{C,C^\ast}:\left(\mathbb{R}^2,g_{C,C^\ast}\right)\to\mathbb{S}^3$, defined by
$$
\Phi_{C,C^\ast}(h,\theta)=\Phi_C(F(h),\theta)=\left(\Phi^1(h),\Phi^2(h), \Phi^3(h)\cos (\sqrt{C}\theta), \Phi^4(h)\sin (\sqrt{C}\theta)\right),
$$
$(h,\theta)\in\mathbb{R}^2$, where $\Phi^1$, $\Phi^2$, $\Phi^3$ and $\Phi^4$ are given by $(\ref{p1})$, $(\ref{p2})$, $(\ref{p3})$ and $(\ref{p4})$, respectively, and the constants $c_k$ are given by $(\ref{c_k2})$, is a biconservative immersion.
\end{theorem}

\begin{proof}
Obviously, for $h\in \left(h_{0,k},h_{0,k+1}\right)$, when $k\geq 1$, or $h\in \left(h_{0,-1},h_{0,1}\right)$, or $h\in \left(h_{0,k-1},h_{0,k}\right)$, when $k\leq -1$, $\Phi_{C,C^\ast}$ is a Riemannian immersion and it is biconservative. As $\Phi_{C,C^\ast}$ is a map of class $C^3$ and the biconservative equation is a third-degree equation, by continuity, we get that $\Phi_{C,C^\ast}$ is biconservative on $\mathbb{R}^2$.
\end{proof}

\begin{remark}
For $C=C^\ast=1$ and $c_0=0$ we obtain the following plot of $\left(\pi\circ\Phi_{1,1}\right)(h,\theta)$, when $h\in\left(h_{0,-11},h_{0,11}\right)$; $\pi:\mathbb{R}^4\to\mathbb{R}^2$ denotes the projection that associates to a vector of $\mathbb{R}^4$ its first two components:

\begin{center}
\begin{tikzpicture}[xscale=1,yscale=1]

\node[inner sep=0pt]  at (.2,0){\includegraphics[width=.5\textwidth]{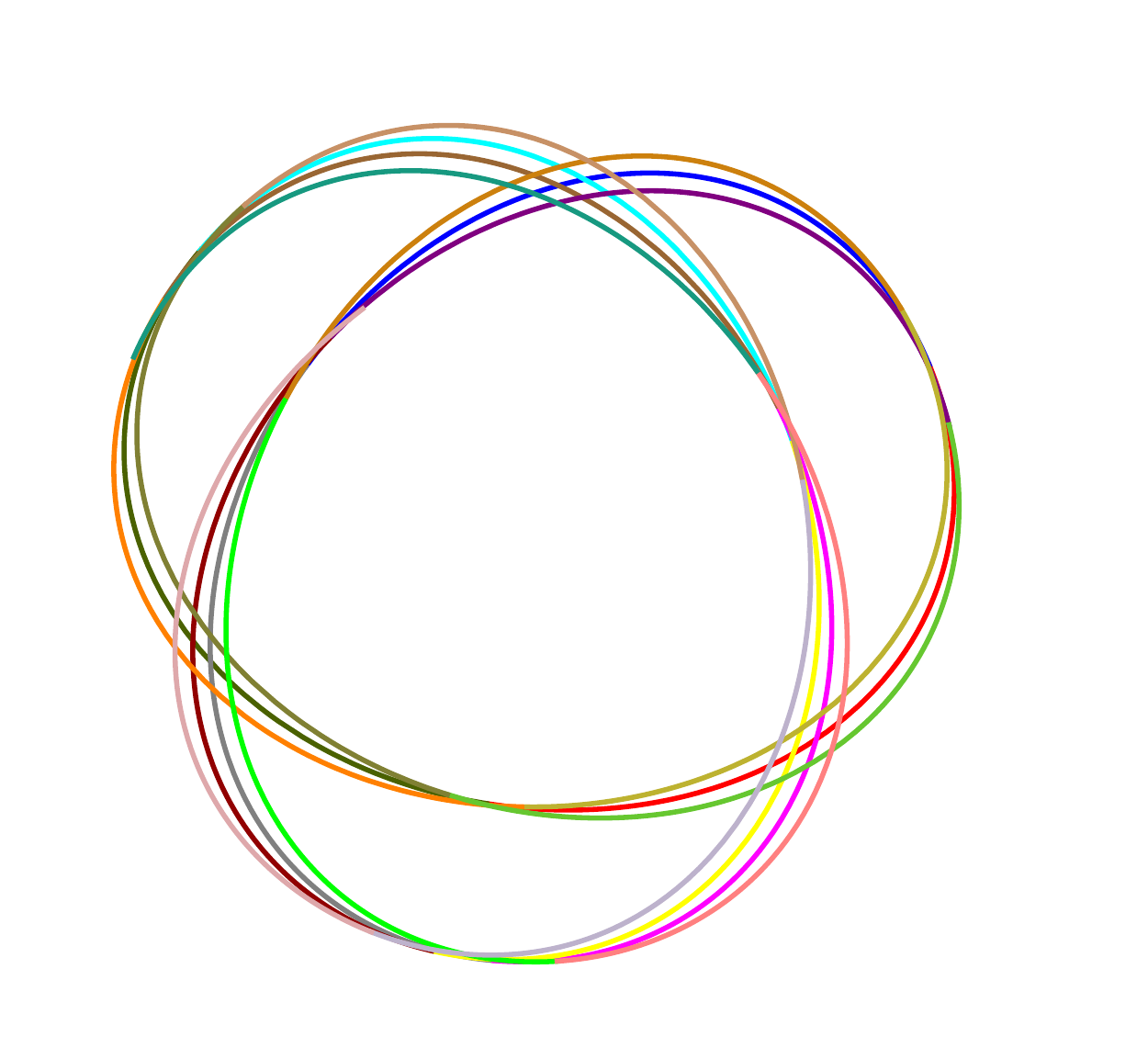}};
\draw [->,>=triangle 45] (-3,0) -- (3,0);
\draw [->,>=triangle 45] (0,-3) -- (0,3);
\draw (3,0) node[below] {$x^1$};
\draw (0,3) node[left] {$x^2$};
\end{tikzpicture}

Figure 6. Plot of $\left(\pi\circ\Phi_{1,1}\right)(h,\theta)$, when $h\in\left(h_{0,-11},h_{0,11}\right)$.

\end{center}

\end{remark}

\begin{remark}
We note that $\Phi_{C,C^\ast}$ has self-intersections (along circles).
\end{remark}

\begin{proposition}
The complete biconservative surfaces given by Theorem $\ref{main_th2}$ are unique (up to reparameterization).
\end{proposition}

\begin{proof}
We first denote by $S_{C,c_k}$ the surface defined by $\Phi_C:\left(D_C,g_C\right)\to\mathbb{S}^3$. Of course, $S_{C,c_k}$ and $S_{C,c_l}$ are extrinsically isometric.

The boundary of $S_{C,c_k}$ is given by the curves:
{\small{
$$
\left(\sqrt{1-\frac{1}{C\xi_{01}^2}}\cos\left(\zeta_{0,-1}+c_k\right),(-1)^k \sqrt{1-\frac{1}{C\xi_{01}^2}}\sin\left(\zeta_{0,-1}+c_k\right),\frac{\cos\left(\sqrt{C}\theta\right)}{\sqrt{C}\xi_{01}}, \frac{\sin\left(\sqrt{C}\theta\right)}{\sqrt{C}\xi_{01}}\right)
$$
}}
and
{\small{
$$
\left(\sqrt{1-\frac{1}{C\xi_{02}^2}}\cos\left(\zeta_{0,1}+c_k\right), (-1)^k \sqrt{1-\frac{1}{C\xi_{02}^2}}\sin\left(\zeta_{0,1}+c_k\right),\frac{\cos\left(\sqrt{C}\theta\right)}{\sqrt{C}\xi_{02}}, \frac{\sin\left(\sqrt{C}\theta\right)}{\sqrt{C}\xi_{02}}\right).
$$
}}
These curves are two circles in the affine planes
$$
\left(\sqrt{1-\frac{1}{C\xi_{01}^2}}\cos\left(\zeta_{0,-1}+c_k\right), (-1)^k \sqrt{1-\frac{1}{C\xi_{01}^2}}\sin\left(\zeta_{0,-1}+c_k\right),0,0\right)+\Span\left\{\overline{e}_3,\overline{e}_4\right\}
$$
and
$$
\left(\sqrt{1-\frac{1}{C\xi_{02}^2}}\cos\left(\zeta_{0,1}+c_k\right), (-1)^k \sqrt{1-\frac{1}{C\xi_{02}^2}}\sin\left(\zeta_{0,1}+c_k\right),0,0\right)+\Span\left\{\overline{e}_3,\overline{e}_4\right\},
$$
respectively. The radii of these two circles are $\frac{1}{\sqrt{C}\xi_{01}}$ and $\frac{1}{\sqrt{C}\xi_{02}}$, respectively.

If we want to glue two surfaces $S_{C,c_k}$ and $S_{C^\prime,c_l}$ then we must do it only along the boundary, and the proof of this result is similar to the proof of Proposition $\ref{prop2.6}$. This implies that the two affine planes, where the boundaries lie, coincide and $C=C^\prime$. Thus, along the boundary, we can glue surfaces only of type $S_{C,c_k}$ and $S_{C,c_l}$.

If we consider, for example, $S_{C,c_0}$ and $S_{C,c_1}$ and glue them along the boundary
$$
\left(\sqrt{1-\frac{1}{C\xi_{02}^2}}\cos\left(\zeta_{0,1}+c_0\right), \sqrt{1-\frac{1}{C\xi_{02}^2}}\sin\left(\zeta_{0,1}+c_0\right),\frac{\cos\left(\sqrt{C}\theta\right)}{\sqrt{C}\xi_{02}}, \frac{\sin\left(\sqrt{C}\theta\right)}{\sqrt{C}\xi_{02}}\right)
$$
for $S_{C,c_0}$ and
$$
\left(\sqrt{1-\frac{1}{C\xi_{02}^2}}\cos\left(\zeta_{0,1}+c_1\right), -\sqrt{1-\frac{1}{C\xi_{02}^2}}\sin\left(\zeta_{0,1}+c_1\right),\frac{\cos\left(\sqrt{C}\theta\right)}{\sqrt{C}\xi_{02}}, \frac{\sin\left(\sqrt{C}\theta\right)}{\sqrt{C}\xi_{02}}\right)
$$
for $S_{C,c_1}$, we get $c_1\equiv \left(-2\zeta_{0,1}-c_0\right)(\mod 2\pi)$, as we have already seen. Then, at a boundary point, using the coordinates $(h,\theta)$ we get that the tangent plane to the closure $\overline{S}_{C,c_0}$ of $S_{C,c_0}$ is spanned by a vector tangent to the boundary and the vector
$$
\left(-\frac{\xi_{02}^{4/3}}{\sqrt{3\left(C\xi_{02}^2-1\right)}}\sin\left(\zeta_{0,1}+c_0\right), \frac{\xi_{02}^{4/3}}{\sqrt{3\left(C\xi_{02}^2-1\right)}}\cos\left(\zeta_{0,1}+c_0\right),0,0\right).
$$
At the same boundary point, the tangent plane to $\overline{S}_{C,c_1}$ is spanned by a vector tangent to the boundary and the vector
$$
\left(\frac{\xi_{02}^{4/3}}{\sqrt{3\left(C\xi_{02}^2-1\right)}}\sin\left(\zeta_{0,1}+c_1\right), \frac{\xi_{02}^{4/3}}{\sqrt{3\left(C\xi_{02}^2-1\right)}}\cos\left(\zeta_{0,1}+c_1\right),0,0\right).
$$
As $c_1\equiv \left(-2\zeta_{0,1}-c_0\right)(\mod 2\pi)$, the two tangent planes coincide.

However, we must then check that we have a $C^3$  smooth gluing.

\end{proof}
\vspace{1cm}
We end this paper with an open problem.

\vspace{0.5cm}
\noindent \textbf{Open problem.} Is there a biconservative immersion $\Phi:\left(M^2,g\right)\to\mathbb{S}^3$, where $M$ is compact, $1-K>0$ on $M$ and $\grad f$ does not vanish at any point of an open dense subset of $M$?

\vspace{0.5cm}
Since $F$ is periodic, $\left(\mathbb{R}^2,g_{C,C^\ast}\right)$ can be quotient to a torus, but we don't know if $\Phi_{C,C^\ast}$ is periodic. Some numerical experiments suggest that $\Phi_{C,C^\ast}$ would not be periodic.

\end{document}